\documentclass[10pt]{amsart}

\IfFileExists{srcltx.sty}{\usepackage{srcltx}}

\usepackage[latin1]{inputenc}
\usepackage{xspace,amssymb,amsfonts,euscript}
\usepackage{amsthm,amsmath}
\usepackage{palatino}
\usepackage{euscript}
\input xy \xyoption {all}
\usepackage{mathrsfs}

\numberwithin{equation}{section}

\RequirePackage{color}
\definecolor{myred}{rgb}{0.75,0,0}
\definecolor{mygreen}{rgb}{0,0.5,0}
\definecolor{myblue}{rgb}{0,0,0.65}

\RequirePackage{ifpdf}
\ifpdf
  \IfFileExists{pdfsync.sty}{\RequirePackage{pdfsync}}{}
  \RequirePackage[pdftex,
   colorlinks = true,
   urlcolor = myblue, 
   citecolor = mygreen, 
   linkcolor = myred, 
   bookmarksopen=true]{hyperref}
\else
  \RequirePackage[hypertex]{hyperref}
\fi

\RequirePackage{ae, aecompl, aeguill} 

  \def\bg{{\mathfrak b}}  \def\BM{{\mathbb{B}}}

    \def\FM{{\mathbb{F}}}
  \def\gg{{\mathfrak g}}  \def\GM{{\mathbb{G}}}

  \def\ng{{\mathfrak n}}  
    
  \def\pg{{\mathfrak p}}

    \def\SM{{\mathbb{S}}}
  \def\tg{{\mathfrak t}}  \def\TM{{\mathbb{T}}}

    \def\ZM{{\mathbb{Z}}}

\def\BB{{\mathbf B}}

    \def\EC{{\mathcal{E}}}
    \def\FC{{\mathcal{F}}}
\def\GB{{\mathbf G}}    \def\GC{{\mathcal{G}}}
\def\HB{{\mathbf H}}    
    \def\IC{{\mathcal{I}}}

    \def\MC{{\mathcal{M}}}
    \def\NC{{\mathcal{N}}}
    \def\OC{{\mathcal{O}}}
\def\PB{{\mathbf P}}    \def\PC{{\mathcal{P}}}

\def\TB{{\mathbf T}}    \def\TC{{\mathcal{T}}}
\def\UB{{\mathbf U}}

\def\XB{{\mathbf X}}

\def\ES{{\EuScript E}}

\def\IS{{\EuScript I}}
\def\JS{{\EuScript J}}

\def\O{\Omega}

\newcommand{\nc}{\newcommand} \newcommand{\renc}{\renewcommand}

\newcommand{\rdots}{\mathinner{ \mkern1mu\raise1pt\hbox{.}
    \mkern2mu\raise4pt\hbox{.}
    \mkern2mu\raise7pt\vbox{\kern7pt\hbox{.}}\mkern1mu}}

\DeclareMathOperator{\Coh}{Coh}
\DeclareMathOperator{\QCoh}{QCoh}

\def\to{\rightarrow}

\def\longto{\longrightarrow}

\nc{\triright}{\stackrel{[1]}{\to}}
\nc{\longtriright}{\stackrel{[1]}{\longto}}

\nc{\Br}{\mathcal{B}}
\nc{\HotRR}{{}_R\mathcal{K}_R}
\nc{\HotR}{\mathcal{K}_R}
\nc{\excise}[1]{}
\nc{\defect}{\text{df}}
\nc{\h}[1]{\underline{H}_{#1}}

\nc{\Ga}{\mathbb{G}_a} 
\nc{\Gm}{\mathbb{G}_{\mathrm{m}}} 

\nc{\Perv}{{\mathbf{P}}}

\nc{\IH}{{\mathrm{IH}}}

\nc{\ic}{\mathbf{IC}}

\nc{\gl}{{\mathfrak{gl}}}
\renc{\sl}{{\mathfrak{sl}}}
\renc{\sp}{{\mathfrak{sp}}}

\nc{\HBM}{H^{BM}}

\DeclareMathOperator{\For}{For} 
 \DeclareMathOperator{\Hom}{Hom}

\DeclareMathOperator{\Rep}{Rep}

\DeclareMathOperator{\id}{Id}

\newtheorem{thm}{Theorem}[section]
\newtheorem{lem}[thm]{Lemma}

\newtheorem{prop}[thm]{Proposition}
\newtheorem{cor}[thm]{Corollary}

\theoremstyle{definition}

\theoremstyle{remark}
\newtheorem{remark}[thm]{Remark}

\DeclareMathOperator{\Ext}{Ext}

\DeclareMathOperator{\Tilt}{Tilt}

\def\tNC{\widetilde{\NC}}
\def\Flag{\mathscr{B}}

\newcommand{\co}{\mathsf{conv}}
\newcommand{\coo}{\mathsf{conv}^\circ}

\def\GD{\check{G}}

\def\Mod{\mathrm{Mod}}

\newcommand{\tgg}{\widetilde{\gg}}
\newcommand{\tpi}{\widetilde{\pi}}

\makeatletter
\def\lotimes{\@ifnextchar_{\@lotimessub}{\@lotimesnosub}}
\def\@lotimessub_#1{\mathchoice{\mathbin{\mathop{\otimes}^{\mathsf{L}}}_{#1}}%
  {\otimes^{\mathsf{L}}_{#1}}{\otimes^{\mathsf{L}}_{#1}}{\otimes^{\mathsf{L}}_{#1}}}
\def\@lotimesnosub{\mathbin{\mathop{\otimes}^{\mathsf{L}}}}
\makeatother

\newcommand{\Waff}{W_{\mathrm{aff}}}
\newcommand{\Baff}{\BM_{\mathrm{aff}}}
\newcommand{\simto}{\xrightarrow{\sim}}
\newcommand{\Cox}{\mathrm{Cox}}

\def\sheafHom{\mathscr{H} \hspace{-1pt} \mathit{om}}

\newcommand{\Db}{D^{\mathrm{b}}}

\newcommand{\Kb}{K^{\mathrm{b}}}

\newcommand{\Inv}{\mathsf{Inv}}

\newcommand{\vv}{\mathsf{v}}

\newcommand{\dom}{\mathsf{dom}}

\newcommand{\St}{\mathsf{M}}
\newcommand{\Cost}{\mathsf{N}}
\newcommand{\Til}{\mathsf{T}}

\newcommand{\Lder}{\mathsf{L}}
\newcommand{\Rder}{\mathsf{R}}

\begin{document}

\begin{abstract}
In this paper we study Bezrukavnikov's exotic t-structure on the derived category of equivariant coherent sheaves on the Springer resolution of a connected reductive algebraic group defined over a field of positive characteristic with simply-connected derived subgroup.  In particular, we show that the heart of the exotic t-structure is a graded highest weight category, and we study the tilting objects in this heart.  Our main tool is the ``geometric braid group action'' studied by Bezrukavnikov and the second author.
%
\end{abstract}

\title[On the exotic t-structure in positive characteristic]{On the exotic t-structure \\ in positive characteristic}

\author{Carl Mautner}
 \address{Department of Mathematics, University of California, Riverside, CA 92521, USA}
\email{mautner@math.ucr.edu}

\thanks{The material in this article is based upon work supported by the National Science Foundation under Grant No. 0932078 000 while the first author was in residence at the Mathematical Sciences Research Institute in Berkeley, California, during the Fall 2014 semester.  C.M. thanks MSRI and the Max Planck Institut f\"ur Mathematik in Bonn for excellent working conditions.}
  
\author{Simon Riche}
\address{Universit{\'e} Blaise Pascal - Clermont-Ferrand II, Laboratoire de Math{\'e}matiques, CNRS, UMR 6620, Campus universitaire des C{\'e}zeaux, F-63177 Aubi{\`e}re Cedex, France}
\email{simon.riche@math.univ-bpclermont.fr}

\thanks{S.R. was supported by ANR Grants No.~ANR-2010-BLAN-110-02 and ANR-13-BS01-0001-01.}

\maketitle


\section{Introduction}

\subsection{The exotic t-structure}

Let $\GB$ be a connected reductive algebraic group with simply-connected derived subgroup, defined over an algebraically closed field $\FM$ of characteristic $p$.  Let $\Flag$ be its flag variety and $\tNC:=T^*(\Flag)$ be the corresponding Springer resolution. The \emph{exotic t-structure} is a certain t-structure on the category $\Db \Coh^{\GB \times \Gm}(\tNC)$ (defined in terms of an exceptional sequence) introduced by Bezrukavnikov in~\cite{bezru-tilting} in the case $p=0$. In this case, this t-structure was used in the proofs of a conjecture of Humphreys on the cohomology of tilting modules for Lusztig's quantum groups at a root of unity, see~\cite{bezru-tilting}, and of some conjectures of Lusztig on the equivariant $K$-theory of Springer fibers, see~\cite{bm}.\footnote{In~\cite{bm} the term ``exotic t-structure'' is used with a (related but) different meaning; the exotic t-structure of~\cite{bezru-tilting} is called \emph{perversely exotic} in~\emph{loc}.~\emph{cit}.} See also~\cite{dodd} for other properties and applications of this t-structure when $p=0$.


The definition of this t-structure also makes sense when $p>0$. It has appeared recently in~\cite{arider2}, with a condition on $p$, in the construction of an equivalence relating the category $\Db\Coh^{\GB\times \Gm}(\tNC)$ with a category of constructible sheaves on the affine Grassmannian of the Langlands dual complex reductive group. (See~\cite{achar2} for remarks and complements on this work.) Using the equivalence, Achar--Rider deduce some basic properties of the exotic t-structure.  However, to obtain the equivalence, they need to assume that $p$ is large enough so that the spherical parity sheaves on the affine Grassmannian are perverse, which had been shown to hold for $p$ larger than explicit bounds in~\cite{jmw}.

\subsection{Highest weight structure}

In this note, we state the definition and prove the basic properties of the exotic t-structure, in a characteristic-free way and independent of connections to the affine Grassmannian.  Our main motivation is the companion paper~\cite{mr}, where we apply these results to show that the equivalence of categories of~\cite{arider2} (and some extensions thereof) exists under weaker assumptions on $p$.
We are then able to deduce as a corollary that the spherical parity sheaves on the affine Grassmannian are perverse whenever $p$ is a good prime for $\GB$.

In particular, we prove that the heart $\ES^{\GB \times \Gm}(\tNC)$ of the exotic t-structure is a graded highest weight category. In the case $p=0$, this property can be deduced from the relation to perverse sheaves on the affine flag variety of the Langlands dual group $\GD$ due to Bezrukavnikov~\cite{bezru-tilting}. For $p$ greater than explicit bounds, this property is proved in~\cite[Proposition~8.5]{arider2} using the equivalence with constructible sheaves on the affine Grassmannian of $\GD$. Our proof is characteristic-free, much simpler, and geometric.

The main tool in our approach is
the \emph{geometric braid group action} on the category $\Db \Coh^{\GB \times \Gm}(\tNC)$ introduced and studied in~\cite{riche, br}. This tool was not used explicitly in~\cite{bezru-tilting} since it had not been introduced then; the relation with the exotic t-structure was made explicit in~\cite{bm} in the case $p=0$.

\subsection{Tilting objects}

The fact that $\ES^{\GB \times \Gm}(\tNC)$ is a graded highest weight category opens the way to the study of tilting objects in this category. Following arguments of Dodd~\cite{dodd} for $p=0$, we give a ``Bott--Samelson type'' construction for these objects. Assuming that $\GB$ satisfies Jantzen's standard hypotheses (i.e. that $p$ is a good prime for $\GB$ and that the Lie algebra of $\GB$ admits a non-degenerate $\GB$-invariant bilinear form), we show that some of these tilting objects (the ``dominant'' ones) can be described explicitly, and that their standard/costandard multiplicities can be expressed in terms of the corresponding multiplicities for tilting $\GB$-modules. 
We also derive interesting consequences on the structure of the category $\ES^{\GB \times \Gm}(\tNC)$; in particular (under the standard hypotheses) we show that costandard objects are coherent sheaves, and that standard and costandard objects satisfy an $\Ext$-vanishing property in $\Db \Coh(\tNC)$. (The first property is also proved in~\cite{arider2} using the comparison with perverse sheaves. The second one can be deduced from the results of~\cite{bm} in case $p$ is $0$ or bigger than the Coxeter number of $\GB$.)

\subsection{Relation with perverse coherent sheaves}

When $\GB$ satisfies the standard hypotheses, the exotic t-structure on $\Db \Coh^{\GB \times \Gm}(\tNC)$ is closely related to the \emph{perverse coherent} t-structure on $\Db \Coh^{\GB \times \Gm}(\NC)$ studied in~\cite{ab, achar} (where $\NC$ is the nilpotent cone of $\GB$). This relation is explained in detail in~\cite[\S 1.4]{achar2}, so we will not recall it here. Note however that the description of dominant tilting objects in $\ES^{\GB \times \Gm}(\tNC)$ is related to, and was motivated by, a similar description of tilting objects in the category of perverse coherent sheaves obtained in~\cite{minn}.



\subsection{Contents}

In Section~\ref{sec:definitions} we recall the definition of Bezrukavnikov's exotic t-structure in terms of an exceptional sequence.
In Section \ref{sec:standard-costandard} we show that the standard objects have a clean description in terms of the geometric braid group action of~\cite{riche, br}, and we use this description to show that the standard and costandard objects belong to the heart of the exotic t-structure, from which it follows that $\ES^{\GB \times \Gm}(\tNC)$ is a graded highest weight category.  In Section~\ref{sec:tilting} we study the tilting objects in $\ES^{\GB \times \Gm}(\tNC)$; in particular we give a ``Bott--Samelson type'' construction of these objects, and we describe explicitly the indecomposable tilting modules associated with dominant weights in terms of the corresponding tilting $\GB$-modules (when $\GB$ satisfies Jantzen's standard hypotheses).
Finally, Appendix~\ref{sec:appendix} contains a brief review of the definitions and main properties of derived categories of equivariant quasi-coherent sheaves and derived functors between them.

\subsection{Acknowledgements}

We thank Pramod Achar for useful discussions at early stages of this work, and an anonymous referee for his/her careful reading of the paper and helpful suggestions. 

\section{Definitions}
\label{sec:definitions}

\subsection{Notation}
\label{ss:notation-exotic}

Let $\FM$ be an algebraically closed field of characteristic $p \geq 0$. We let $\GB$ be a connected reductive algebraic group over $\FM$ whose derived subgroup is simply connected. Let $\BB \subset \GB$ be a Borel subgroup, and $\TB \subset \BB$ be a maximal torus.
Let $\tg \subset \bg \subset \gg$ be the Lie algebras of $\TB \subset \BB \subset \GB$. Let also $\UB$ be the unipotent radical of $\BB$, and $\ng$ be its Lie algebra. 

For some of our results we will have to assume that $p$ is a good prime for $\GB$ and that there exists a non-degenerate $\GB$-invariant bilinear form on $\gg$. 
In this case we say that $\GB$ is \emph{standard}.

We will consider the varieties
\[
\Flag := \GB/\BB \quad \text{and} \quad \tNC:=T^*(\Flag).
\]
(Here $\Flag$ is the flag variety of $\GB$, a smooth projective variety over $\FM$, and $\tNC$ is the \emph{Springer resolution} of the nilpotent cone of $\GB$.\footnote{More precisely of the nilpotent cone $\NC^*$ in the dual of $\gg$.})
The scheme $\tNC$ is endowed with a natural action\footnote{The action of $\Gm$ considered here agrees with the action considered in~\cite{riche, mr, arider2}, but differs from the action considered in~\cite{bezru-tilting, br} by the automorphism $t \mapsto t^{-1}$ of $\Gm$.} of $\GB \times \Gm$, where the action of $\GB$ is induced by the natural action on $\Flag$, and $t \in \Gm$ acts by multiplication by $t^{-2}$ on the fibers of the projection $\tNC \to \Flag$.
Moreover, there exists a $\GB \times \Gm$-equivariant isomorphism
\[
\tNC \cong \GB \times^{\BB} (\gg/\bg)^*,
\]
where $t \in \Gm$ acts on $(\gg/ \bg)^*$ by multiplication by $t^{-2}$.

We will consider the derived categories of coherent sheaves
\[
D^{\GB \times \Gm}(\tNC):=\Db\Coh^{\GB \times \Gm}(\tNC), \qquad D^{\GB}(\tNC):=\Db\Coh^{\GB}(\tNC).
\]
We denote by
\[
\langle 1 \rangle \colon D^{\GB \times \Gm}(\tNC) \simto D^{\GB \times \Gm}(\tNC)
\]
the functor of tensoring with the one-dimensional tautological $\Gm$-module, and by $\langle n \rangle$ its $n$-th power.\footnote{This notation agrees with~\cite{bezru-tilting, riche, br, mr}, but is opposite to the convention in~\cite{arider2}.} Then for any $\FC,\GC$ in $D^{\GB \times \Gm}(\tNC)$ the forgetful functor induces an isomorphism
\begin{equation}
\label{eqn:Hom-Gm-equiv}
\bigoplus_{m \in \ZM} \Hom_{D^{\GB \times \Gm}(\tNC)}(\FC, \GC \langle m \rangle) \simto \Hom_{D^{\GB}(\tNC)}(\FC, \GC).
\end{equation}

Let $\XB:=X^*(\TB)$, resp.~${\check \XB}:=X_*(\TB)$, be the weight lattice, resp.~the coweight lattice, and $\Phi \subset \XB$, resp.~${\check \Phi} \subset {\check \XB}$, be the roots, resp.~coroots, of $\GB$ (with respect to $\TB$). We let $\ZM \Phi \subset \XB$ be the root lattice. The choice of $\BB$ determines a system of positive roots: more precisely we denote by $\Phi^+ \subset \Phi$ the roots which are \emph{opposite} to the $\TB$-weights in $\bg$. We set $\Phi^- := -\Phi^+$, and denote by ${\check \Phi}^+,{\check \Phi}^-$ the positive and negative coroots, respectively. If $\alpha \in \Phi$, we denote by $\alpha^\vee$ the associated coroot. We let $\rho \in \XB$ 
be a weight such that $\langle \rho, \alpha^\vee \rangle = 1$ for all simple roots $\alpha$. (If $\GB$ is semisimple then $\rho$ is the half sum of positive roots.)


The choice of $\Phi^+$ determines a subset $\XB^+ \subset \XB$ of dominant weights, and a partial order $\preceq$ on $\XB$. Each $\lambda \in \XB$ also determines a line bundle $\OC_{\Flag}(\lambda)$ on $\Flag$. With our conventions, $\OC_{\Flag}(\lambda)$ is ample iff $\lambda \in \XB^+$. We denote by $\OC_{\tNC}(\lambda)$ the pullback of $\OC_{\Flag}(\lambda)$ to $\tNC$.

For $\lambda \in \XB$, we denote by $\co(\lambda)$ the intersection of the convex hull of $\lambda$ with $\lambda + \ZM \Phi$, and by $\coo(\lambda)$ the complement of $W\lambda$ in $\co(\lambda)$. If $\HB$ is any $\FM$-algebraic group, we denote by $\Rep(\HB)$ the category of algebraic $\HB$-modules.

\begin{remark}
In this paper we make use of results from~\cite{riche, br} where it is assumed that $\GB$ is semisimple and simply-connected. However, one can check that all the results we use apply equally well when $\GB$ is a connected reductive group whose derived subgroup is simply connected.  The condition on the derived subgroup is needed 
in particular
in order to ensure that the affine braid group admits the Bernstein presentation described below, which is used 
in~\cite{riche, br}
to define the action of the affine braid group on $D^{\GB \times \Gm}(\tNC)$.
\end{remark}

\subsection{Affine braid group}
\label{ss:Haff}


Let $W$ be the Weyl group of $(\GB,\TB)$, and let $\Waff := W \ltimes \XB$ be the (extended) affine Weyl group. To avoid confusion, for $\lambda \in \XB$ we denote by $t_\lambda$ the corresponding element of $\Waff$.  The subgroup $\Waff^\Cox:=W \ltimes (\ZM \Phi) \subset \Waff$ is a Coxeter group; we choose the Coxeter generators as in~\cite[\S 1.4]{IM}.
The simple reflections which belong to $W$ will be called \emph{finite}; the ones which do not belong to $W$ will be called \emph{affine}. If $\alpha$ is a simple root, we denote by $s_\alpha$ the corresponding (finite) simple reflection.

The length function for our Coxeter structure on $\Waff^\Cox$ satisfies the following formula for $w \in W$ and $\lambda \in \ZM \Phi$ (see \cite[Proposition~1.10]{IM}):
\begin{equation}
\label{eqn:length} 
\ell(w \cdot t_\lambda)=\sum_{\alpha \in \Phi^+ \cap w^{-1} (\Phi^+)} |\langle \lambda,
\alpha^\vee \rangle | + \sum_{\alpha \in
      \Phi^+ \cap w^{-1}(\Phi^-)}
  |1 + \langle \lambda, \alpha^\vee \rangle |.
\end{equation}
We use this formula to extend $\ell$ to the whole of $\Waff$. We denote by $\O$ the subgroup of $\Waff$ consisting of elements of length $0$; it is a finitely generated abelian group isomorphic to $\XB/ \ZM\Phi$ via the composition of natural maps $\O \hookrightarrow W \ltimes \XB \twoheadrightarrow \XB \twoheadrightarrow \XB/\ZM\Phi$. Moreover, the conjugation action of $\Omega$ on $\Waff$ preserves $\Waff^\Cox$ (it sends any simple reflection to a simple reflection), and multiplication induces a group isomorphism $\Omega \ltimes \Waff^\Cox \simto \Waff$.

The following is an easy consequence of~\eqref{eqn:length}; details are left to the reader.

\begin{lem}
\label{lem:length-translations}
For any $\lambda \in \XB$ and $w \in W$ we have
$\ell(t_{w\lambda}) = \ell(t_\lambda)$.\qed
\end{lem}

We will also consider the braid group $\Baff$ associated with $\Waff$. It is defined as the group generated by elements $T_w$ for $w \in \Waff$, with relations $T_{vw} = T_v T_w$ for all $v,w \in \Waff$ such that $\ell(vw)=\ell(v)+\ell(w)$. One can define (following Bernstein and Lusztig), for each $\lambda \in \XB$, an element $\theta_\lambda \in \Baff$, see e.g.~\cite[\S 1.1]{riche} for details.  When $\lambda$ is dominant, $\theta_\lambda$ is simply $T_{t_\lambda}$. The affine braid group $\Baff$ admits a second useful presentation (usually called the \emph{Bernstein presentation}), with generators $\{T_w, \ w \in W\}$ and $\{\theta_\lambda, \ \lambda \in \XB\}$, subject to the following relations (where $v,w \in W$, $\lambda, \mu \in \XB$, and $\alpha$ runs over simple roots):
\begin{enumerate}
\item
$T_v T_w = T_{vw}$ \quad if $\ell(vw)=\ell(v)+\ell(w)$;
\item
$\theta_\lambda \theta_\mu = \theta_{\lambda+\mu}$;
\item
$T_{s_\alpha} \theta_\lambda = \theta_\lambda T_{s_\alpha}$ \quad if $\langle \lambda, \alpha^\vee \rangle = 0$;
\item
$\theta_\lambda = T_{s_\alpha} \theta_{\lambda-\alpha} T_{s_\alpha}$ \quad if $\langle \lambda, \alpha^\vee \rangle = 1$.
\end{enumerate}
(See~\cite{br-appendix} for a proof of this fact.)

\begin{lem}
\label{lem:T-t-lambda}
Let $\lambda \in \XB$ and $w \in W$ be such that $w\lambda$ is dominant. Then we have
\[
T_{t_\lambda} = T_{w^{-1}} \cdot \theta_{w\lambda} \cdot (T_{w^{-1}})^{-1} \qquad \text{in $\Baff$.}
\]
\end{lem}

\begin{proof}
We have $t_\lambda w^{-1} = w^{-1} t_{w\lambda}$ in $\Waff$, and $\ell(w^{-1} t_{w\lambda}) = \ell(w^{-1}) + \ell(t_{w\lambda})$ since $w\lambda$ is dominant (see~\eqref{eqn:length}). Using Lemma~\ref{lem:length-translations} we deduce that we also have $\ell(t_\lambda w^{-1})=\ell(t_\lambda) + \ell(w^{-1})$. Hence in $\Baff$ we have
\[
T_{t_\lambda w^{-1}} = T_{t_\lambda} \cdot T_{w^{-1}} = T_{w^{-1}} \cdot T_{t_{w\lambda}}.
\]
Now, since $w\lambda$ is dominant we have $T_{t_{w\lambda}} = \theta_{w\lambda}$, and the lemma follows.
\end{proof}

For $\lambda \in \XB$, we denote by $w_\lambda$ the shortest representative in $W t_\lambda \subset \Waff$, and by $\delta(\lambda)$ the minimal length of an element $v \in W$ such that $v\lambda$ is dominant. The following lemma is probably well known, but we were not able to find a reference in the literature.

\begin{lem}
\label{lem:w-lambda}
Let $\lambda \in \XB$, and let $v \in W$ be of minimal length such that $v\lambda$ is dominant. Then we have
\[
w_\lambda = t_{v\lambda} v = v t_\lambda \qquad \text{and} \qquad \ell(w_\lambda) = \ell(t_\lambda) - \delta(\lambda).
\]
\end{lem}

\begin{proof}
We prove the claim by induction on $\delta(\lambda)$. If $\delta(\lambda)=0$, i.e.~if $\lambda$ is dominant, then from~\eqref{eqn:length} we deduce that $w_\lambda=t_\lambda$, and we are done.

Now suppose that $\lambda$ is not dominant. Write $v=us$, where $u \in W$ and $s$ is a finite simple reflection such that $u < v$. Then by induction we have $w_{s\lambda} = t_{v\lambda} u = u t_{s\lambda}$. The element $w_{s\lambda} s = v t_\lambda$ belongs to $W t_\lambda$. If $w_{s\lambda} s$ is not minimal in $W t_\lambda$ then, as Soergel observes in~\cite[%
p.~86]{soergel-kl}, there exists a finite simple reflection $r$ such that both $rw_{s\lambda} > w_{s\lambda}$ and $rw_{s\lambda}s < w_{s\lambda} s$.  Taken together, these imply that
 $w_{s\lambda} s = r w_{s\lambda}$. Since the left-hand side belongs to $Wt_\lambda$ and the right-hand side to $W t_{s\lambda}$, it follows that $\lambda=s\lambda$, contradicting the minimality of $v$. This proves the induction step for the first equality.

To prove the second equality, we use~\eqref{eqn:length} to obtain that
\[
\ell(w_\lambda)=\sum_{\alpha \in \Phi^+ \cap v^{-1} ( \Phi^+)} |\langle \lambda,
  \alpha^{\vee} \rangle | + \sum_{\alpha \in \Phi^+ \cap v^{-1}(\Phi^-)}
  |1 + \langle \lambda, \alpha^\vee \rangle |
\]
and
\[
\ell(w_{s\lambda})=\sum_{\alpha \in \Phi^+ \cap sv^{-1} (\Phi^+)} |\langle s\lambda,
  \alpha^{\vee} \rangle | + \sum_{\alpha \in \Phi^+ \cap sv^{-1}(\Phi^-)}
  |1 + \langle s\lambda, \alpha^{\vee} \rangle |.
\]
One can rewrite the second equality as
\[
\ell(w_{s\lambda})=\sum_{\alpha \in s( \Phi^+) \cap v^{-1} ( \Phi^+)} |\langle \lambda,
  \alpha^{\vee} \rangle | + \sum_{\alpha \in s(\Phi^+) \cap v^{-1}(\Phi^-)}
  |1 + \langle \lambda, \alpha^{\vee} \rangle |.
\]
Now we observe that if $\beta$ is the simple root associated with $s$, then $s(\Phi^+)=(\Phi^+ \smallsetminus \{\beta\}) \cup \{- \beta\}$, and that $\beta \in v^{-1}( \Phi^-)$ since $vs<v$. Using also the fact that $\langle \lambda, \beta^{\vee} \rangle < 0$, this implies that $\ell(w_\lambda)=\ell(w_{s\lambda})-1$, which completes the induction step.
\end{proof}




\subsection{Reminder on graded exceptional sequences}
\label{ss:reminder}

Let $\Bbbk$ be a field, 
and let $\mathsf{D}$ be a $\Bbbk$-linear triangulated category endowed with an autoequivalence $X \mapsto X \langle 1 \rangle$. For any $m \in \ZM$, we denote by $X \mapsto X \langle m \rangle$ the $m$-th power of this autoequivalence. We assume that $\mathsf{D}$ 
is of graded finite type, i.e.~such that for any $X,Y$ in $\mathsf{D}$ the $\Bbbk$-vector space $\bigoplus_{n,m} \Hom^n(X,Y \langle m \rangle)$ is finite dimensional. Then a collection $\{\nabla^i, \, i \in I\}$ of objects of $\mathsf{D}$ (where $(I,\leq)$ is a partially ordered set) is called a \emph{graded exceptional sequence} if it satisfies
\[
\Hom^n(\nabla^i, \nabla^j \langle m \rangle)=0
\]
if $i \not\geq j$ or if $i=j$ and $(n,m) \neq (0,0)$, and moreover $\Hom(\nabla^i, \nabla^i) = \Bbbk$.

From now on we fix a graded exceptional sequence $\{\nabla^i, \, i \in I\}$.
If $(I, \leq)$ is 
isomorphic to a disjoint union of copies of $\ZM_{\geq 0}$, then there exists a unique collection $\{\Delta^i, \, i \in I\}$ of objects of $\mathsf{D}$ which satisfy
\[
\Hom^n(\Delta^i, \nabla^j \langle m \rangle)=0 \ \text{if $i>j$} \qquad \text{and} \qquad \Delta^i \cong \nabla^i \mod \mathsf{D}_{<i}
\]
(where $\mathsf{D}_{<i}$ is the full triangulated subcategory of $\mathsf{D}$ generated by the objects $\nabla^j \langle m \rangle$ for $j<i$ and $m \in \ZM$, and the right-hand side means that the images of $\Delta^i$ and $\nabla^i$ in the Verdier quotient $\mathsf{D}/\mathsf{D}_{<i}$ are isomorphic), see~\cite[Proposition~3]{bezru-tilting}. This collection is called the \emph{dual graded exceptional sequence}; it is a graded exceptional sequence for $I$ equipped with the order opposite to $\leq$.
These objects automatically satisfy the condition
\[
\Hom^n(\Delta^i, \nabla^j \langle m \rangle) = \begin{cases}
\Bbbk & \text{if $i=j$ and $n=m=0$;} \\
0 & \text{otherwise.}
\end{cases}
\]


\begin{lem}
\label{lem:standard-Hom-gr}
Assume that the objects $\{\nabla^j \langle m \rangle, \, j \in I, \, m \in \ZM\}$ generate $\mathsf{D}$ as a triangulated category, and that $(I, \leq)$ is 
isomorphic to a disjoint union of copies of $\ZM_{\geq 0}$. If $i \in I$ and $X \in \mathsf{D}$ satisfy
\[
\Hom^n(X, \nabla^j \langle m \rangle) = \begin{cases}
\Bbbk & \text{if $i=j$ and $n=m=0$;} \\
0 & \text{otherwise,}
\end{cases}
\]
then $X \cong \Delta^i$.
\end{lem}

\begin{proof}
Let us choose non-zero morphisms $\varphi \colon \Delta^i \to \nabla^i$ and $\psi \colon X \to \nabla^i$, and let $C$ denote the cone of $\varphi$. Then $C$ belongs to $\mathsf{D}_{<i}$. Applying the functor $\Hom(X,-)$ to the distinguished triangle $\Delta^i \xrightarrow{\varphi} \nabla^i \to C \xrightarrow{[1]}$ we obtain an exact sequence
\[
\Hom^{-1}(X,C) \to \Hom(X, \Delta^i) \to \Hom(X, \nabla^i) \to \Hom(X,C).
\]
The condition on $X$ implies that the first and fourth terms vanish. Hence $\psi$ factors through a morphism $\psi' \colon X \to \Delta^i$. It is easy to check that the cone $C'$ of $\psi'$ satisfies $\Hom^n(C', \nabla^j \langle m \rangle)=0$ for all $j \in I$ and all $n,m \in \ZM$. Using our first assumption, it follows that $C'=0$, i.e.~that $\psi'$ is an isomorphism.
\end{proof}

Assume as above that the objects $\{\nabla^j \langle m \rangle, \, j \in I, \, m \in \ZM\}$ generate $\mathsf{D}$ as a triangulated category, and that $(I, \leq)$ is 
isomorphic to a disjoint union of copies of $\ZM_{\geq 0}$.
Define $\mathsf{D}^{\geq 0}$ as the full subcategory of $\mathsf{D}$ generated under extensions by the objects $\nabla^i \langle m \rangle [n]$ with $i \in I$, $m \in \ZM$ and $n \in \ZM_{\leq 0}$, and $\mathsf{D}^{\leq 0}$ as the full subcategory of $\mathsf{D}$ generated under extensions by the objects $\Delta^i \langle m \rangle [n]$ with $i \in I$, $m \in \ZM$ and $n \in \ZM_{\geq 0}$. Then by~\cite[Proposition~4]{bezru-tilting}, the pair $(\mathsf{D}^{\leq 0},\mathsf{D}^{\geq 0})$ is a bounded t-structure on $\mathsf{D}$, called the t-structure associated with the graded exceptional sequence $\{\nabla^i, \, i \in I\}$. Note that the functor $\langle 1 \rangle$ is t-exact.


\subsection{Line bundles}


The following lemma is well known.

\begin{lem}
\label{lem:line-bundles}
Let $\lambda, \mu \in \XB$ and $n,m \in \ZM$.
\begin{enumerate}
\item
\label{it:Hom-line-bundles}
We have $\Hom^n_{D^{\GB \times \Gm}(\tNC)}(\OC_{\tNC}(\lambda), \OC_{\tNC}(\mu) \langle m \rangle)=0$ unless $\mu \preceq \lambda$. 
\item
\label{it:Hom-O}
We have 
\[
\Hom^n_{D^{\GB \times \Gm}(\tNC)}(\OC_{\tNC}, \OC_{\tNC} \langle m \rangle) = \begin{cases}
\FM & \text{if $n=m=0$;} \\
0 & \text{otherwise.}
\end{cases}
\]
\item
\label{it:Hom-finite-dim}
The $\FM$-vector space
\[
\bigoplus_{n,m \in \ZM} \Hom^n_{D^{\GB \times \Gm}(\tNC)}(\OC_{\tNC}(\lambda), \OC_{\tNC}(\mu) \langle m \rangle)
\]
is finite-dimensional.
\end{enumerate}
\end{lem}

\begin{proof}
Statement
\eqref{it:Hom-line-bundles} can be proved 
by the same arguments as in the proof of~\cite[Lemma 1.11.8]{br}. Now we consider~\eqref{it:Hom-O}. Using~\eqref{eqn:Hom-Gm-equiv}, it is enough to prove a similar claim for morphisms in $D^\GB(\tNC)$. Recall that we have a natural isomorphism $\tNC \cong \GB \times^{\BB} (\gg/\bg)^*$.
We deduce an equivalence of categories $D^{\GB}(\tNC) \cong \Db \Coh^{\BB}( (\gg/\bg)^* )$, hence an isomorphism
\[
\Hom^n_{D^{\GB}(\tNC)}(\OC_{\tNC}, \OC_{\tNC}) \cong \Hom^n_{\Db \Coh^{\BB}( (\gg/\bg)^* )}(\OC_{(\gg/\bg
)^*}, \OC_{(\gg/\bg)^*}).
\]
The right-hand side is isomorphic to $\Hom^n_{\mathsf{Rep}(\BB)}(\FM, \mathrm{S}(\gg/\bg))$, where $\mathrm{S}(\gg/\bg)$ is the symmetric algebra of the $\BB$-module $\gg/\bg$. Now if $\mathrm{S}^+(\gg/\bg)$ denotes the kernel of the natural morphism $\mathrm{S}(\gg/\bg) \to \FM$, then
$\Hom^m_{\mathsf{Rep}(\BB)}(\FM, \mathrm{S}^+(\gg/\bg))=0$ for all $m$, by~\cite[Proposition~II.4.10(b)]{jantzen}. We deduce an isomorphism
\[
\Hom^n_{\mathsf{Rep}(\BB)}(\FM, \mathrm{S}(\gg/\bg)) \cong \Hom^n_{\mathsf{Rep}(\BB)}(\FM, \FM).
\]
The right-hand side is described in~\cite[Corollary~II.4.11]{jantzen}, and the proof of~\eqref{it:Hom-O} is complete.

Finally, we consider~\eqref{it:Hom-finite-dim}. Using again~\eqref{eqn:Hom-Gm-equiv}, what we have to prove is that the vector space
\[
\bigoplus_{n \in \ZM} \Hom^n_{D^{\GB}(\tNC)}(\OC_{\tNC}(\lambda), \OC_{\tNC}(\mu))
\]
is finite dimensional. As above, this amounts to showing that the vector space
\[
\bigoplus_{n \in \ZM} \Hom^n_{\mathsf{Rep}(\BB)}(\FM, \mathrm{S}(\gg/\bg) \otimes \FM_{\mu-\lambda})
\]
is finite dimensional. However, there are only finitely many $\TB$-weights in $\mathrm{S}(\gg/\bg) \otimes \FM_{\mu-\lambda}$ whose opposite is a sum of positive roots. Hence, using again~\cite[Proposition~II.4.10(b)]{jantzen}, it suffices to prove that for any $\nu \in \XB$ the vector space
\[
\bigoplus_{n \in \ZM} \Hom^n_{\mathsf{Rep}(\BB)}(\FM, \FM_{\nu})
\]
is finite dimensional. The latter fact follows from~\cite[Proposition~II.4.10]{jantzen}.
\end{proof}

\begin{cor}
\label{cor:exceptional-sequence}
\begin{enumerate}
\item
\label{it:generation}
The objects $\{\OC_{\tNC}(\lambda) \langle n \rangle, \lambda \in \XB, n \in \ZM\}$ generate the triangulated category $D^{\GB \times \Gm}(\tNC)$.
\item
\label{it:finite-type}
The category $D^{\GB \times \Gm}(\tNC)$ is of graded finite type.
\item
\label{it:exceptional-sequence}
The objects $\{\OC_{\tNC}(\lambda), \lambda \in \XB\}$ of $D^{\GB \times \Gm}(\tNC)$ form a graded exceptional sequence para\-metrized by the partially ordered set $(\XB, \preceq)$.
\end{enumerate}
\end{cor}

\begin{proof}
Statement~\eqref{it:generation} is proved in~\cite[Corollary~5.8]{achar}.\footnote{In~\cite{achar} it is assumed that $p$ is good, and $\tNC$ is defined as $\GB \times^\BB \ng$ rather than $\GB \times^\BB (\gg/\bg)^*$. However, one can easily check that the proof of~\cite[Corollary~5.8]{achar} applies verbatim in our setting.} Then~\eqref{it:finite-type} follows from~\eqref{it:generation} and Lemma~\ref{lem:line-bundles}\eqref{it:Hom-finite-dim}. And~\eqref{it:exceptional-sequence} is a restatement of Lemma~\ref{lem:line-bundles}\eqref{it:Hom-line-bundles}--\eqref{it:Hom-O}.
\end{proof}

\subsection{Exotic t-structure}
\label{ss:exotic-t-structure}

As in~\cite[\S 2.3]{bezru-tilting} we define another order $\leq$ on $\XB$, as follows: we set $\lambda \leq \mu$ if $w_\lambda$ preceeds $w_\mu$ in the Bruhat order (where $w_\lambda$ is defined in~\S\ref{ss:Haff}). One can easily check that this order coincides with $\preceq$ on $\XB^+$, and on any $W$-orbit in $\XB$. It is also clear that if $\lambda - \mu \notin \ZM \Phi$, then $\lambda$ and $\mu$ are not comparable.

For each $\ZM \Phi$-coset $\Lambda$ in $\XB$ we choose a refinement $\leq'$ of $\leq$ on $\Lambda$ which satisfies the following conditions:
\begin{itemize}
\item
the ordered set $(\Lambda, \leq ')$ is isomorphic to $\ZM_{\geq 0}$ (with its standard order);
\item
if $\lambda, \mu \in \Lambda$ and $\lambda \in \coo(\mu)$, then $\lambda <' \mu$.
\end{itemize}
We then define the order $\leq '$ on $\XB$ by gluing these orders on each coset. (Note that this order is not a total order in general: elements in different $\ZM \Phi$-cosets are incomparable.)

We let $\{\nabla^\lambda_0, \, \lambda \in \XB\}$ be the graded exceptional sequence in $D^{\GB \times \Gm}(\tNC)$ which is the $\leq'$-mutation of the sequence of Corollary~\ref{cor:exceptional-sequence}\eqref{it:exceptional-sequence} in the sense of \cite[\S 2.1.4]{bezru-tilting}.
Recall that $\nabla^\lambda_0$ is defined as the unique object (up to a unique isomorphism) such that 
\[
\nabla^\lambda_0 \in  D^{\GB \times \Gm}_{\leq' \lambda}(\tNC) \cap  D^{\GB \times \Gm}_{<' \lambda}(\tNC)^\perp\quad \text{and} \quad \nabla^\lambda_0 \cong \OC_{\tNC}(\lambda) \mod  D^{\GB \times \Gm}_{<' \lambda}(\tNC).
\]
(Here we write $D^{\GB \times \Gm}_{\leq' \lambda}(\tNC)$ for $\bigl( D^{\GB \times \Gm}(\tNC) \bigr)_{\leq' \lambda}$, with the notation of~\S\ref{ss:reminder}. And if $C$ is a subcategory of an additive category $D$, we use the notation $C^\perp$ to denote the full subcategory of $D$ consisting of objects $\FC$ such that $\Hom(\GC,\FC) =0$ for all $\GC$ in $C$.)
We denote by $\{\Delta^\lambda_0, \, \lambda \in \XB\}$ the dual graded exceptional sequence.

Bezrukavnikov's \emph{exotic t-structure} is then defined to be the bounded t-structure on $D^{\GB \times \Gm}(\tNC)$ associated with the graded exceptional sequence $\{\nabla^\lambda_0, \, \lambda \in \XB\}$ (see~\S\ref{ss:reminder}).
The heart of this t-structure will be denoted by $\ES^{\GB\times \Gm}(\tNC)$.
The objects $\nabla^\lambda_0 \langle m \rangle$, resp.~$\Delta^\lambda_0 \langle m \rangle$ for $\lambda \in \XB$ and $m \in \ZM$ will be called \emph{costandard objects}, resp.~\emph{standard objects}. These objects satisfy
\begin{equation}
\label{eqn:Hom-vanishing-DN}
\Hom^n(\Delta^\lambda_0, \nabla^\mu_0 \langle m \rangle) =
\begin{cases}
\FM & \text{if $\lambda=\mu$ and $n=m=0$;} \\
0 & \text{otherwise.}
\end{cases}
\end{equation}

\begin{remark}
In this paper we only consider the exotic t-structure on $D^{\GB \times \Gm}(\tNC)$. A similar construction yields a bounded t-structure on $D^{\GB}(\tNC)$, such that the forgetful functor $D^{\GB \times \Gm}(\tNC) \to D^{\GB}(\tNC)$ is t-exact. The results of Sections~\ref{sec:standard-costandard}--\ref{sec:tilting} have obvious analogues in this context, which we will not state. (See~\cite{bezru-tilting} for the case $p=0$.)
\end{remark}

\section{Further study of standard and costandard objects}
\label{sec:standard-costandard}

\subsection{Description of costandard objects}
\label{ss:costandard}

Following~\cite[\S 2.3]{bezru-tilting} one can give a more concrete description of the exceptional sequence $\{\nabla^\lambda_0, \, \lambda \in \XB\}$, which we explain in this subsection.

Let us fix a simple root $\alpha$, and let $s:=s_\alpha$. We consider the subvariety $Z_s' \subset \tNC \times \tNC$ defined in \cite[\S 1.3]{br}. This subvariety has two irreducible components: the diagonal $\Delta \tNC$, and another which we denote $Y_s$; moreover there exist short exact sequences in $\Db \Coh^{\GB \times \Gm}(\tNC \times \tNC)$ (where $\GB \times \Gm$ acts diagonally on $\tNC \times \tNC$):
\begin{gather}
\OC_{Y_s}(-\rho, \rho-\alpha) \hookrightarrow \OC_{Z_s'} \twoheadrightarrow \OC_{\Delta \tNC}, \label{eqn:ses-tNC-1} \\
\OC_{\Delta \tNC} \langle 2 \rangle \hookrightarrow \OC_{Z_s'}(-\rho, \rho-\alpha) \twoheadrightarrow \OC_{Y_s}(-\rho, \rho-\alpha),\label{eqn:ses-tNC-2}
\end{gather}
where in each sequence the surjection is induced by restriction of functions (see \cite[Lemma 6.1.1]{riche}).
Here $\OC_{Z_s'}(-\rho, \rho-\alpha)$ is the tensor product of $\OC_{Z_s'}$ with the line bundle $\OC_{\tNC}(-\rho) \boxtimes \OC_{\tNC}(\rho-\alpha)$ on $\tNC \times \tNC$, and similarly for $\OC_{Y_s}(-\rho, \rho-\alpha)$. Note that $\OC_{Z_s'}(-\rho, \rho-\alpha) \cong \OC_{Z_s'}(\rho-\alpha, -\rho)$ and similarly for $Y_s$; see~\cite[Lemma~1.5.1 and comments thereafter]{riche}.

By~\cite[Proposition~1.10.3]{br}, the object $\OC_{Z_s'} \langle -1 \rangle$ in $\Db \Coh^{\GB \times \Gm}(\tNC \times \tNC)$ is invertible for the convolution product, with inverse $\OC_{Z_s'}(-\rho, \rho-\alpha) \langle -1 \rangle$. We denote by 
\[
\TM_{\tNC}^s \colon D^{\GB \times \Gm}(\tNC) \to D^{\GB \times \Gm}(\tNC), \quad \text{resp.} \quad \SM^s_{\tNC} \colon D^{\GB \times \Gm}(\tNC) \to D^{\GB \times \Gm}(\tNC),
\]
the Fourier--Mukai transform with kernel $\OC_{Z_s'} \langle -1 \rangle$, resp.~$\OC_{Z_s'}(-\rho, \rho-\alpha) \langle -1 \rangle$, as defined in \cite[\S 1.2]{br}. Then we have
$\TM_{\tNC}^s \circ \SM_{\tNC}^s \cong \SM_{\tNC}^s \circ \TM_{\tNC}^s \cong \id$.
(See Appendix~\ref{sec:appendix} for details on definitions of usual functors in the equivariant setting.)

For $\lambda \in \XB$, we denote by $D^{\GB \times \Gm}_{\co(\lambda)}(\tNC)$ the full triangulated subcategory of the category $D^{\GB \times \Gm}(\tNC)$ generated by the line bundles $\OC_{\tNC}(\mu) \langle m \rangle$ for $\mu \in \co(\lambda)$ and $m \in \ZM$. We define the subcategory $D^{\GB \times \Gm}_{\coo(\lambda)}(\tNC)$ in a similar way. The following lemma can be proved as in~\cite[Lemma 1.11.3]{br} (or can be deduced from the latter result); details are left to the reader. (See also~\cite[Lemma 7]{bezru-tilting} for a similar statement.)

\begin{lem}
\label{lem:TM-SM-line-bundles}
\begin{enumerate}
\item
\label{it:T-S-line-bundles-2}
If $\lambda \in \XB$ and $\lambda=s\lambda$ then
\[
\TM_{\tNC}^s(\OC_{\tNC}(\lambda)) \cong \OC_{\tNC}(s\lambda) \langle -1 \rangle \quad \text{and} \quad \SM_{\tNC}^s(\OC_{\tNC}(\lambda)) \cong \OC_{\tNC}(s\lambda) \langle 1 \rangle.
\]
\item
\label{it:T-S-line-bundles-3}
If $\lambda \in \XB$ and $\lambda \prec s\lambda$, then 
\[
\TM_{\tNC}^s(\OC_{\tNC}(\lambda)) \cong \OC_{\tNC}(s\lambda) \langle 1 \rangle \mod D^{\GB \times \Gm}_{\coo(\lambda)}(\tNC).
\]
\item
\label{it:T-S-line-bundles-4}
If $\lambda \in \XB$ and $s\lambda \prec \lambda$ then 
\[
\SM_{\tNC}^s(\OC_{\tNC}(\lambda)) \cong \OC_{\tNC}(s\lambda) \langle -1 \rangle \mod D^{\GB \times \Gm}_{\coo(\lambda)}(\tNC).
\]
\item
\label{it:T-S-line-bundles-1}
For $\lambda \in \XB$, the functors $\TM_{\tNC}^s$ and $\SM^s_{\tNC}$ stabilize the subcategories $D^{\GB \times \Gm}_{\coo(\lambda)}(\tNC)$ and $D^{\GB \times \Gm}_{\co(\lambda)}(\tNC)$.\qed
\end{enumerate}
\end{lem}


By \cite[Proposition~3]{bezru-tilting}, the inclusion $D^{\GB \times \Gm}_{\coo(\lambda)}(\tNC)^\bot \to D^{\GB \times \Gm}(\tNC)$ admits a left adjoint
\[
\Pi^l_\lambda \colon D^{\GB \times \Gm}(\tNC) \to D^{\GB \times \Gm}_{\coo(\lambda)}(\tNC)^\bot.
\] 
More concretely (see e.g.~\cite[Lemma~3.1]{bondal}), for any $X$ in $D^{\GB \times \Gm}(\tNC)$ there exist unique objects $Y$ in $D^{\GB \times \Gm}_{\coo(\lambda)}(\tNC)$ and $Z$ in $D^{\GB \times \Gm}_{\coo(\lambda)}(\tNC)^\perp$ and a unique distinguished triangle
$Y \to X \to Z \xrightarrow{[1]}$
(up to unique isomorphisms); then $Z=\Pi^l_\lambda(X)$.

For $\lambda \in \XB$, we set $(\nabla_0')^\lambda := \Pi^l_\lambda(\OC_{\tNC}(\lambda))$. (This notation will be used only in this subsection.)

\begin{lem}
\label{lem:SM-Pi-O}
Let $\lambda \in \XB$.
\begin{enumerate}
\item
\label{it:S-nabla-1}
If $s\lambda=\lambda$, then $\SM_{\tNC}^s((\nabla_0')^\lambda) \cong (\nabla_0')^{s\lambda} \langle 1 \rangle$.
\item
\label{it:S-nabla-2}
If $s\lambda \prec \lambda$, then $\SM_{\tNC}^s((\nabla_0')^\lambda) \cong (\nabla_0')^{s\lambda} \langle -1 \rangle$.
\end{enumerate}
\end{lem}

\begin{proof}
We prove~\eqref{it:S-nabla-2}; the proof of~\eqref{it:S-nabla-1} is similar. The arguments are copied from~\cite[p.~344]{bezru-tilting}.

It is easily checked that the composition of natural functors $D^{\GB \times \Gm}_{\coo(\lambda)}(\tNC)^\perp \to D^{\GB \times \Gm}(\tNC) \to D^{\GB \times \Gm}(\tNC) / D^{\GB \times \Gm}_{\coo(\lambda)}(\tNC)$ is an equivalence, see e.g.~\cite[Proposition~1.6]{bonkap}. Therefore, it
suffices to prove the following properties: 
\begin{gather}
\SM_{\tNC}^s((\nabla_0')^\lambda) \cong \OC_{\tNC}(s\lambda) \langle -1 \rangle \mod D^{\GB \times \Gm}_{\coo(\lambda)}(\tNC); \label{eqn:nabla'-1} \\ 
\SM_{\tNC}^s((\nabla_0')^\lambda) \in D^{\GB \times \Gm}_{\coo(\lambda)}(\tNC)^\perp. \label{eqn:nabla'-2}
\end{gather}

By definition, $(\nabla_0')^\lambda \cong \OC_{\tNC}(\lambda) \mod D^{\GB \times \Gm}_{\coo(\lambda)}(\tNC)$.  Applying $\SM_{\tNC}^s$ and using Lemma~\ref{lem:TM-SM-line-bundles}\eqref{it:T-S-line-bundles-1}, we obtain
\[
\SM_{\tNC}^s((\nabla_0')^\lambda) \cong \SM_{\tNC}^s(\OC_{\tNC}(\lambda)) \mod D^{\GB \times \Gm}_{\coo(\lambda)}(\tNC).
\]
Then, using
Lemma~\ref{lem:TM-SM-line-bundles}\eqref{it:T-S-line-bundles-4}, we deduce~\eqref{eqn:nabla'-1}.

For any $X \in D^{\GB \times \Gm}_{\coo(\lambda)}(\tNC)$, we observe that
\[
\Hom(X, \SM_{\tNC}^s((\nabla_0')^\lambda)) \cong \Hom(\TM_{\tNC}^s(X), (\nabla_0')^\lambda )=0,
\]
because, by definition, $(\nabla_0')^\lambda \in D^{\GB \times \Gm}_{\coo(\lambda)}(\tNC)^\perp$ and, by Lemma~\ref{lem:TM-SM-line-bundles}\eqref{it:T-S-line-bundles-1}, $\TM_{\tNC}^s(X) \in D^{\GB \times \Gm}_{\coo(\lambda)}(\tNC)$. This proves~\eqref{eqn:nabla'-2}.
\end{proof}

\begin{prop}
\label{prop:nabla}
Let $\lambda \in \XB$.
\begin{enumerate}
\item 
\label{it:nabla-1}
We have $\nabla^\lambda_0 \cong \Pi^l_\lambda(\OC_{\tNC}(\lambda))$.
\item
\label{it:nabla-2}
If $s\lambda = \lambda$, then $\SM_{\tNC}^s(\nabla_0^\lambda) \cong \nabla_0^{s\lambda} \langle 1 \rangle$.
\item
\label{it:nabla-Ss}
If $s\lambda \prec \lambda$, then $\SM_{\tNC}^s(\nabla_0^\lambda) \cong \nabla_0^{s\lambda} \langle -1 \rangle$.
\end{enumerate}
\end{prop}

\begin{proof}
These claims are proved in \cite[Proposition 7]{bezru-tilting} when $p=0$. The arguments are similar in our setting, but for completeness we repeat them below. First we observe that it suffices to prove~\eqref{it:nabla-1}; then~\eqref{it:nabla-2} and~\eqref{it:nabla-Ss} follow from Lemma~\ref{lem:SM-Pi-O}.


We wish to show that $(\nabla_0')^\lambda = \nabla_0^\lambda$.
To do so, it suffices to prove the following properties:
\begin{gather} 
(\nabla_0')^\lambda \cong \OC_{\tNC}(\lambda) \mod D^{\GB \times \Gm}_{<' \lambda}(\tNC); \label{eqn:nabla-nabla'-1} \\
\Hom^n ((\nabla_0')^\lambda, (\nabla_0')^\mu \langle m \rangle)= 0 \quad \text{for any $n,m \in \ZM$ and $\lambda <' \mu$}. \label{eqn:nabla-nabla'-2}
\end{gather}

By definition, $(\nabla_0')^\lambda \cong \OC_{\tNC}(\lambda) \mod D^{\GB \times \Gm}_{\coo(\lambda)}(\tNC)$.  As $\nu \in \coo(\lambda)$ implies
that $\nu <' \lambda$ by definition (see~\S\ref{ss:exotic-t-structure}), condition~\eqref{eqn:nabla-nabla'-1} is immediate.
In order to verify~\eqref{eqn:nabla-nabla'-2}, we will prove the stronger property that 
\begin{equation}
\label{eqn:nabla-nabla'-3}
\Hom^n ((\nabla_0')^\lambda, (\nabla_0')^\mu \langle m \rangle)= 0 \quad \text{if $\mu \in W \lambda$ and $\mu \not\preceq \lambda$, or if $\mu \notin \co(\lambda)$}.
\end{equation}

First, assume that $\mu \in W \lambda$ and $\mu \not\preceq \lambda$.  Then $\coo(\lambda)=\coo(\mu)$, hence $(\nabla_0')^\mu \in D^{\GB \times \Gm}_{\coo(\lambda)}(\tNC)^\perp$.  Thus we may apply adjunction:
\[
\Hom^n ((\nabla_0')^\lambda, (\nabla_0')^\mu \langle m \rangle) \cong \Hom^n (\OC_{\tNC}(\lambda), (\nabla_0')^\mu \langle m \rangle).
\]
We claim that $(\nabla_0')^\mu$ belongs to the triangulated subcategory of $ D^{\GB \times \Gm}(\tNC)$ generated by the objects $\OC_{\tNC}(\nu) \langle p \rangle$ with $\mu \preceq \nu$: in fact this follows from the explicit construction of $\Pi_\mu^l$ in~\cite[Theorem~3.2(b)]{bondal} and Lemma~\ref{lem:line-bundles}\eqref{it:Hom-line-bundles}.
Having assumed $\mu \not\preceq \lambda$, we conclude (again by Lemma~\ref{lem:line-bundles}\eqref{it:Hom-line-bundles}) that 
\[
\Hom^n (\OC_{\tNC}(\lambda), (\nabla_0')^\mu \langle m \rangle) = 0,
\]
proving~\eqref{eqn:nabla-nabla'-3} in this case.

Now assume instead that $\mu \notin \co(\lambda)$.  Write $\mu = w \mu^+$, where $w\in W$ and $\mu^+ \in \XB^+$.  As $(\nabla_0')^\lambda \in D^{\GB \times \Gm}_{\co(\lambda)}(\tNC)$ by construction, it suffices to show that $(\nabla_0')^\mu \in D^{\GB \times \Gm}_{\co(\lambda)}(\tNC)^\perp$ under our assumption.  We proceed by induction on $\ell(w)$.

If $\ell(w)=0$, i.e.~$\mu \in \XB^+$, then by Lemma~\ref{lem:line-bundles}\eqref{it:Hom-line-bundles} we have $\OC_{\tNC}(\mu) \in D^{\GB \times \Gm}_{\coo(\mu)}(\tNC)^\perp$ and so $(\nabla_0')^\mu = \OC_{\tNC}(\mu)$. On the other hand we also have $\OC_{\tNC}(\mu) \in D^{\GB \times \Gm}_{\co(\lambda)}(\tNC)^{\perp}$ by the same Lemma~\ref{lem:line-bundles}\eqref{it:Hom-line-bundles}.

Now assume that $(\nabla_0')^{\mu'} \in D^{\GB \times \Gm}_{\co(\lambda)}(\tNC)^\perp$ and $\mu=s\mu' \prec \mu'$, where $s$ is a finite simple reflection. Then $(\nabla_0')^{\mu} = \SM_{\tNC}^s ((\nabla_0')^{\mu'}) \langle 1 \rangle$ by Lemma~\ref{lem:SM-Pi-O}\eqref{it:S-nabla-2}.  But $\SM_{\tNC}^s$ preserves the category $D^{\GB \times \Gm}_{\co(\lambda)}(\tNC)^\perp$, as $\SM_{\tNC}^s$ is right adjoint to $\TM_{\tNC}^s$, which preserves $D^{\GB \times \Gm}_{\co(\lambda)}(\tNC)$ by Lemma~\ref{lem:TM-SM-line-bundles}\eqref{it:T-S-line-bundles-1}.
\end{proof}

\begin{cor}
\label{cor:nabla-dominant}
If $\lambda \in \XB^+$ then $\nabla_0^\lambda=\OC_{\tNC}(\lambda)$.
\end{cor}

\begin{proof}
We have observed in the course of the proof of Proposition~\ref{prop:nabla} that $(\nabla'_0)^\lambda \cong \OC_{\tNC}(\lambda)$. Since $(\nabla'_0)^\lambda \cong \nabla^\lambda_0$ by Proposition~\ref{prop:nabla}\eqref{it:nabla-1}, the corollary follows.
\end{proof}

\begin{remark}
Proposition~\ref{prop:nabla} and Corollary~\ref{cor:nabla-dominant} entirely determine the objects $\nabla_0^\lambda$, for all $\lambda \in \XB$. This description implies in particular that these objects do not depend on the choice of the order $\leq'$. Then Lemma~\ref{lem:standard-Hom-gr} shows that the objects $\Delta^0_\lambda$ do not depend on this choice either.
\end{remark}

\subsection{Description of standard objects}
\label{ss:standard}

One can give a description of standard objects which is similar to the description of costandard objects given in~\S\ref{ss:costandard}. This description is suggested in~\cite{bezru-tilting}, though it is not stated explicitly.
As in~\S\ref{ss:costandard} we fix a simple root $\alpha$, and set $s:=s_\alpha$.

\begin{prop}
\label{prop:Delta}
Let $\lambda \in \XB$.
\begin{enumerate}
\item
\label{it:Delta-Ss}
If $\lambda \preceq s \lambda$, then $\SM_{\tNC}^s(\Delta^\lambda_0) \cong \Delta^{s\lambda}_0 \langle 1 \rangle$.
\item
\label{it:Delta-antidom}
If $\lambda \in -\XB^+$ then $\Delta^\lambda_0 \cong \OC_{\tNC}(\lambda)$.
\end{enumerate}
\end{prop}

\begin{proof}
\eqref{it:Delta-Ss} For $m,n \in \ZM$ we have
\[
\Hom^n \bigl( \SM_{\tNC}^s(\Delta^\lambda_0), \nabla^{s\lambda}_0 \langle m \rangle \bigr) \cong \Hom^n \bigl( \Delta^\lambda_0, \TM_{\tNC}^s(\nabla^{s\lambda}_0) \langle m \rangle \bigr).
\]
If $\lambda=s\lambda$, then by Proposition~\ref{prop:nabla}(2) we have $ \TM_{\tNC}^s(\nabla^{s\lambda}_0) \cong \nabla^\lambda_0 \langle -1 \rangle$, and we deduce that
\begin{equation}
\label{eqn:Delta-S-1}
\Hom^n \bigl( \SM_{\tNC}^s(\Delta^\lambda_0), \nabla^{s\lambda}_0 \langle m \rangle \bigr) \cong 
\begin{cases}
\FM & \text{if $n=0$ and $m=1$;}\\
0 & \text{otherwise.}
\end{cases}
\end{equation}

Now we assume that $\lambda \prec s\lambda$.
Using exact sequences~\eqref{eqn:ses-tNC-1} and~\eqref{eqn:ses-tNC-2} we obtain distinguished triangles
\[
\FC \to \TM^s_{\tNC}(\nabla^{s\lambda}_0) \to \nabla^{s\lambda}_0 \langle -1 \rangle \xrightarrow{[1]} \quad \text{and} \quad \nabla^{s\lambda}_0 \langle 1 \rangle \to \SM^s_{\tNC}(\nabla^{s\lambda}_0) \to \FC \xrightarrow{[1]}
\]
for some object $\FC$ in $D^{\GB \times \Gm}(\tNC)$. Since $\SM^s_{\tNC}(\nabla^{s\lambda}_0) \cong \nabla^\lambda_0 \langle -1 \rangle$ by Proposition~\ref{prop:nabla}\eqref{it:nabla-Ss}, using the second triangle and~\eqref{eqn:Hom-vanishing-DN} we obtain that $\Hom^n(\Delta^\lambda_0, \FC \langle m \rangle)$ is isomorphic to $\FM$ if $n=0$ and $m=1$, and $0$ otherwise.
Then using the first triangle we deduce that, similarly,
$\Hom^n \bigl( \Delta^\lambda_0, \TM_{\tNC}^s(\nabla^{s\lambda}_0) \langle m \rangle \bigr)$ is isomorphic to $\FM$ if $n=0$ and $m=1$, and $0$ otherwise. Finally, we obtain that~\eqref{eqn:Delta-S-1} also holds in this case.

Now let $\mu \in \XB \smallsetminus \{ s \lambda \}$. We want to show that
\begin{equation}
\label{eqn:Delta-S-2}
\Hom^n \bigl( \SM_{\tNC}^s(\Delta^\lambda_0), \nabla_0^{\mu} \langle m \rangle \bigr) =0 \quad \text{for all $n,m \in \ZM$.}
\end{equation}
As above
we have
\[
\Hom^n \bigl( \SM_{\tNC}^s(\Delta^\lambda_0), \nabla_0^{\mu} \langle m \rangle \bigr) \cong \Hom^n \bigl( \Delta^\lambda_0, \TM_{\tNC}^s(\nabla_0^{\mu}) \langle m \rangle \bigr).
\]
If $\mu \prec s\mu$ then by Proposition~\ref{prop:nabla}\eqref{it:nabla-Ss} we have $\TM_{\tNC}^s(\nabla^{\mu}_0) \cong \nabla_0^{s\mu} \langle 1 \rangle$, and we deduce~\eqref{eqn:Delta-S-2} since $\lambda \neq s\mu$. The case $\mu=s\mu$ is similar, using Proposition~\ref{prop:nabla}\eqref{it:nabla-2}.
If $s \mu \prec \mu$ (which implies $\mu \neq \lambda$) then using the exact sequences~\eqref{eqn:ses-tNC-1} and~\eqref{eqn:ses-tNC-2} we obtain triangles
\[
\GC \to \TM_{\tNC}^s (\nabla_0^{\mu}) \to \nabla_0^{\mu} \langle -1 \rangle \xrightarrow{[1]} \quad \text{and} \quad \nabla_0^{\mu} \langle 1 \rangle \to \SM_{\tNC}^s (\nabla_0^{\mu}) \to \GC \xrightarrow{[1]}
\]
for some object $\GC$ in $D^{\GB \times \Gm}(\tNC)$. Since $\SM_{\tNC}^s (\nabla_0^{\mu}) \cong \nabla_0^{s \mu} \langle -1 \rangle$ (again by Proposition~\ref{prop:nabla}\eqref{it:nabla-Ss}) and since $\lambda \notin \{\mu, s\mu\}$ we deduce that~\eqref{eqn:Delta-S-2} also holds in this case.

Finally the claim follows from Lemma~\ref{lem:standard-Hom-gr}, using~\eqref{eqn:Delta-S-1} and \eqref{eqn:Delta-S-2}.

\eqref{it:Delta-antidom}
Assume that $\lambda \in -\XB^+$. By definition of the dual exceptional sequence (see~\S\ref{ss:reminder}), it suffices to prove that
\[
\Hom^n(\OC_{\tNC}(\lambda), \nabla_0^\mu \langle m \rangle) = 0 \quad \text{if $\mu <' \lambda$.}
\]
Let $\mu \in \XB$, and assume that $\Hom^n(\OC_{\tNC}(\lambda), \nabla_0^\mu \langle m \rangle) \neq 0$ for some $n,m \in \ZM$. By Proposition~\ref{prop:nabla}\eqref{it:nabla-1} we know that $\nabla_0^\mu \in D^{\GB \times \Gm}_{\co(\mu)}(\tNC)$. Using Lemma~\ref{lem:line-bundles}\eqref{it:Hom-line-bundles} we deduce that there exists $\nu \in \co(\mu)$ such that $\nu \preceq \lambda$, which implies that $\lambda \in \co(\mu)$ since $\lambda \in -\XB^+$.
If $\lambda \in W\mu$ then $\lambda \leq' \mu$. If $\lambda \in \coo(\mu)$ then $\lambda <' \mu$ by definition of the order $\leq'$ (see~\S\ref{ss:exotic-t-structure}). In any case one cannot have $\mu <' \lambda$, which finishes the proof.
\end{proof}

\subsection{Affine braid group action}
\label{ss:braid-action-Springer}

By~\cite[Theorem 1.6.1]{br} there exists
a \emph{right} action\footnote{Here by a (left) action of a group on a category we mean a group morphism from the given group to the group of \emph{isomorphism classes} of autoequivalences of our category. As usual, a right action of a group is a left action of the opposite group.} of $\BM_{\mathrm{aff}}$ on $D^{\GB \times \Gm}(\tNC)$, where we denote by 
$\JS_b$ 
the action of $b \in \BM_{\mathrm{aff}}$ (defined up to isomorphism), and which is defined on generators by
\[
\JS_{T_s} \cong \TM^s_{\tNC}, \qquad \JS_{\theta_\lambda} \cong \OC_{\tNC}(\lambda) \otimes_{\OC_{\tNC}} (-).
\]
(The action considered in~\cite{br} is a \emph{left} action; but the relations in the presentation of $\Baff$ given in~\cite[\S 1.1]{br} are symmetric, so that checking the relations for a right action is the same as checking the relations for a left action. In other words, the action considered in the present paper is deduced from the action considered in~\cite{br} by composition with the unique anti-automorphism of $\Baff$ fixing the generators $T_s$ for $s$ a finite simple reflection and $\theta_\lambda$ for $\lambda \in \XB$.)

The results of \S\S\ref{ss:costandard}--\ref{ss:standard} have the following consequence.

\begin{prop}
\label{prop:nabla-delta-braid-group}
If $\lambda \in \XB$ and $w \in W t_\lambda \subset \Waff$, then we have
\[
\nabla^\lambda_0 \cong \JS_{T_w}(\OC_{\tNC}) \langle \ell(w) - \ell(t_\lambda) + 2 \delta(\lambda) \rangle, \quad \Delta^\lambda_0 \cong \JS_{(T_{w^{-1}})^{-1}}(\OC_{\tNC}) \langle -\ell(w) + \ell(t_\lambda) \rangle.
\]
\end{prop}

\begin{proof}
Using the formula
\begin{equation}
\label{eqn:J-O}
\JS_{T_v}(\OC_{\tNC}) \cong  \OC_{\tNC} \langle -\ell(v) \rangle
\end{equation}
for $v \in W$ (see Lemma~\ref{lem:TM-SM-line-bundles}(2)), it is enough to prove each isomorphism for one element in $Wt_\lambda$, for instance for $t_\lambda$.

Let us consider the first isomorphism.
Let $v \in W$ be of minimal length such that $v \lambda \in \XB^+$. By Lemma~\ref{lem:T-t-lambda}, we have $T_{t_\lambda} = T_{v^{-1}} \cdot \theta_{v\lambda} \cdot (T_{v^{-1}})^{-1}$. We deduce that
\[
 \JS_{T_{t_\lambda}}(\OC_{\tNC}) \cong \JS_{(T_{v^{-1}})^{-1}} \JS_{\theta_{v \lambda}} \JS_{T_{v^{-1}}}(\OC_{\tNC}) \cong \JS_{(T_{v^{-1}})^{-1}} \bigl( \OC_{\tNC}(v\lambda) \bigr) \langle -\ell(v) \rangle,
\]
where the last isomorphism uses~\eqref{eqn:J-O}.
By Corollary~\ref{cor:nabla-dominant}, we have $\OC_{\tNC}(v \lambda) \cong \nabla^{v \lambda}_0$. Hence if $v=s_1 \cdots s_r$ is a reduced expression for $v \in W$, we obtain an isomorphism
\[
\JS_{T_{t_\lambda}}(\OC_{\tNC}) \cong \SM_{\tNC}^{s_r} \circ \cdots \circ \SM_{\tNC}^{s_1}(\nabla_0^{v \lambda}) \langle -r \rangle.
\]
Now we have
\[
\lambda \prec s_r \lambda \prec \cdots \prec s_1v\lambda \prec v\lambda,
\]
hence using Proposition~\ref{prop:nabla}\eqref{it:nabla-Ss} repeatedly, we deduce an isomorphism
\[
\JS_{T_{t_\lambda}}(\OC_{\tNC}) \cong \nabla_0^\lambda \langle -2r \rangle.
\]
Since $r=\delta(\lambda)$, this finishes the proof of the first isomorphism.

Now we prove the second isomorphism when $w=t_\lambda$. Let $v \in W$ be of minimal length such that $v\lambda$ is antidominant. Then using Lemma~\ref{lem:T-t-lambda} we obtain that
\[
T_{t_\lambda^{-1}} = T_{t_{-\lambda}} = T_{v^{-1}} \cdot \theta_{-v\lambda} \cdot (T_{v^{-1}})^{-1}.
\]
Hence
\[
\JS_{(T_{t_\lambda^{-1}})^{-1}}(\OC_{\tNC}) \cong \JS_{(T_{v^{-1}})^{-1}} \JS_{\theta_{v\lambda}} \JS_{T_{v^{-1}}}(\OC_{\tNC}) \cong \JS_{(T_{v^{-1}})^{-1}} \bigl( \OC_{\tNC}(v\lambda) \bigr) \langle -\ell(v) \rangle
\]
(where the last isomorphism uses~\eqref{eqn:J-O}). Since $v\lambda$ is antidominant, by Proposition~\ref{prop:Delta}\eqref{it:Delta-antidom} we have $\OC_{\tNC}(v\lambda) \cong \Delta_0^{v\lambda}$. Choose a reduced expression $v=s_1 \cdots s_r$. Then we obtain that
\[
\JS_{(T_{t_\lambda^{-1}})^{-1}}(\OC_{\tNC}) \cong \SM_{\tNC}^{s_r} \circ \cdots \circ \SM_{\tNC}^{s_1}(\Delta_0^{v\lambda}) \langle -r \rangle.
\]
Since $v \lambda \prec s_1v \lambda \prec \cdots \prec s_r \lambda \prec \lambda$, a repeated application of Proposition~\ref{prop:Delta}\eqref{it:Delta-Ss} allows to conclude that $\JS_{(T_{t_\lambda^{-1}})^{-1}}(\OC_{\tNC}) \cong \Delta_0^\lambda$, which finishes the proof.
\end{proof}

\begin{remark}
\label{rmk:standard-action}
\begin{enumerate}
\item
It follows from Proposition~\ref{prop:nabla-delta-braid-group} that the costandard, resp.~standard, objects are exactly the objects of the form $\JS_{T_w}(\OC_{\tNC}) \langle m \rangle$, resp.~of the form $\JS_{T_w^{-1}}(\OC_{\tNC}) \langle m \rangle$, for $w \in \Waff$ and $m \in \ZM$. 
\item
Formulas similar to those of Proposition~\ref{prop:nabla-delta-braid-group} appear (in characteristic zero, and without the $\Gm$-equivariance) in~\cite[Proof of Lemma~6.2.4]{bm}.
\end{enumerate}
\end{remark}

It is convenient (see~\cite{achar2, mr}) to use a slightly different normalization of the standard and costandard objects, setting
\[
\nabla^\lambda_{\tNC} := \JS_{T_{w_\lambda}}(\OC_{\tNC}) \cong \nabla^\lambda_0 \langle -\delta(\lambda) \rangle, \qquad
\Delta^\lambda_{\tNC} := \JS_{(T_{w_\lambda^{-1}})^{-1}}(\OC_{\tNC}) \cong \Delta^\lambda_0 \langle -\delta(\lambda) \rangle.
\]
(In these formulas we use Lemma~\ref{lem:w-lambda}.)

\subsection{Standard and costandard objects are in the heart}

The main result of this subsection is that the objects $\Delta^\lambda_{\tNC}$ and $\nabla^\lambda_{\tNC}$ belong to $\ES^{\GB \times \Gm}(\tNC)$, see Corollary~\ref{cor:nabla-delta-heart}. Our proof was inspired by a proof of the similar claim for (co-)standard objects in the mixed derived category of a flag variety in an earlier version of~\cite{modrap2}.

\begin{prop}
\label{prop:exactness}
For any $w \in \Waff$, the functor $\JS_{T_w}$ is right t-exact, and the functor $\JS_{T_w^{-1}}$ is left t-exact.
\end{prop}

\begin{proof}
It suffices to prove the claim when $\ell(w) \in \{0,1\}$. The case $\ell(w)=0$ is easy: in fact, since in this case $T_w T_{w'} = T_{ww'}$ for all $w' \in \Waff$, it follows from Remark~\ref{rmk:standard-action} that the functor $\JS_{T_w}$ sends every standard, resp.~costandard, object to a standard, resp.~costandard, object. Hence it is t-exact. And since $T_w^{-1} = T_{w^{-1}}$, $\JS_{T_w^{-1}}$ is also t-exact.

Now, assume that $s$ is a simple reflection (finite or affine). By Remark~\ref{rmk:standard-action}, any standard object is of the form $\JS_{T_w^{-1}}(\OC_{\tNC}) \langle m \rangle$ for some $w \in \Waff$ and $m \in \ZM$. Assume first that $\ell(sw)<\ell(w)$. Then we have $T_w = T_s T_{sw}$, so that the object
\[
\JS_{T_s} \bigl( \JS_{T_w^{-1}}(\OC_{\tNC}) \langle m \rangle \bigr) \cong \JS_{T_s} \JS_{T_s^{-1}} \JS_{T_{sw}^{-1}} (\OC_{\tNC}) \langle m \rangle \cong \JS_{T_{sw}^{-1}} (\OC_{\tNC}) \langle m \rangle
\]
is standard, hence in particular in $\bigl( D^{\GB \times \Gm}(\tNC) \bigr)^{\leq 0}$. On the other hand, if $\ell(sw)>\ell(w)$ then we have $T_s T_w = T_{sw}$, hence
\[
\JS_{T_s} \bigl( \JS_{T_w^{-1}}(\OC_{\tNC}) \langle m \rangle \bigr) = \JS_{T_s} \JS_{T_s} \JS_{T_{sw}^{-1}} (\OC_{\tNC}) \langle m \rangle.
\]
Using exact sequences~\eqref{eqn:ses-tNC-1} and~\eqref{eqn:ses-tNC-2} we obtain distinguished triangles
\begin{gather*}
\FC \to \JS_{T_s} \JS_{T_s} \JS_{T_{sw}^{-1}} (\OC_{\tNC}) \langle m \rangle \to \JS_{T_s} \JS_{T_{sw}^{-1}} (\OC_{\tNC}) \langle m \rangle \xrightarrow{[1]}, \\
\JS_{T_s} \JS_{T_{sw}^{-1}} (\OC_{\tNC}) \langle m+2 \rangle \to \JS_{T_{sw}^{-1}} (\OC_{\tNC}) \langle m \rangle \to \FC \xrightarrow{[1]}
\end{gather*}
for some object $\FC$ in $D^{\GB \times \Gm}(\tNC)$. (More precisely, the exact sequences exist only when $s$ is a \emph{finite} simple reflection. But if $s$ is affine, then by~\cite[Lemma~6.1.2]{riche2} $T_s$ is conjugate in $\BM_{\mathrm{aff}}$ to $T_t$ for some \emph{finite} simple reflection $t$; hence in this case, by conjugating the exact sequences associated with $t$ by the appropriate kernel, we obtain distinguished triangles which play a role similar to the exact sequences~\eqref{eqn:ses-tNC-1} and~\eqref{eqn:ses-tNC-2}.) Since both $\JS_{T_{sw}^{-1}} (\OC_{\tNC})$ and $\JS_{T_s} \JS_{T_{sw}^{-1}} (\OC_{\tNC}) \cong \JS_{T_w^{-1}}(\OC_{\tNC})$ are standard objects, the second triangle implies that $\FC$ is in $\bigl( D^{\GB \times \Gm}(\tNC) \bigr)^{\leq 0}$. Then the first triangle implies that $\JS_{T_s} \JS_{T_s} \JS_{T_{sw}^{-1}} (\OC_{\tNC}) \langle m \rangle$ is also in $\bigl( D^{\GB \times \Gm}(\tNC) \bigr)^{\leq 0}$, which finishes the proof of the right t-exactness of $\JS_{T_s}$.

Finally, to prove that $\JS_{T_s^{-1}}$ is left t-exact, we simply remark that this functor has a right t-exact left adjoint (namely its inverse $\JS_{T_s}$); hence it is indeed left t-exact.
\end{proof}

\begin{cor}
\label{cor:nabla-delta-heart}
The objects $\nabla^\lambda_{\tNC}$ and $\Delta^\lambda_{\tNC}$ belong to $\ES^{\GB \times \Gm}(\tNC)$.\end{cor}

\begin{proof}
By Propositions \ref{prop:nabla-delta-braid-group} and \ref{prop:exactness} we obtain that $\nabla^\lambda_{\tNC} \in \bigl( D^{\GB\times\GM}(\tNC) \bigr)^{\leq 0}$ (since $\OC_{\tNC}=\Delta^0_{\tNC}=\nabla^0_{\tNC}$ belongs to $\ES^{\GB \times \Gm}(\tNC)$). By definition this object belongs to $\bigl( D^{\GB\times\GM}(\tNC) \bigr)^{\geq 0}$; hence it is in $\ES^{\GB \times \Gm}(\tNC)$.
Similar arguments show that $\Delta^\lambda_{\tNC}$ also belongs to $\ES^{\GB \times \Gm}(\tNC)$.
\end{proof}

\subsection{Highest weight structure}
\label{ss:qh-structure}

It follows from Corollary~\ref{cor:nabla-delta-heart} that the abelian category $\ES^{\GB \times \Gm}(\tNC)$, endowed with the collections of objects $\{\Delta^\lambda_{\tNC}, \, \lambda \in \XB\}$ and $\{\nabla^\lambda_{\tNC}, \, \lambda \in \XB\}$, is graded highest weight, in the sense that the poset $\XB$ is interval-finite and the category satisfies axioms (2)--(5) of~\cite[Definition~A.1]{modrap2}.\footnote{Cline, Parshall and Scott introduced the notion of a highest weight category in~\cite[Definition 3.1]{cps-fdalg} and of a positively graded highest weight category in~\cite[(1.2)]{cps-dual}.}
In particular, it makes sense to consider the \emph{tilting objects} in $\ES^{\GB \times \Gm}(\tNC)$, i.e.~the objects which admit both a standard filtration (i.e.~a filtration with subquotients which are standard objects)
and a costandard filtration (i.e.~a filtration with subquotients which are costandard objects).
The subcategory consisting of such objects will be denoted by $\Tilt(\ES^{\GB \times \Gm}(\tNC))$.

If $X$ admits a standard filtration, resp.~a costandard filtration, we denote by $(X : \Delta^\lambda_{\tNC} \langle m \rangle)$, resp.~$(X : \nabla^\lambda_{\tNC} \langle m \rangle)$, the number of times $\Delta^\lambda_{\tNC} \langle m \rangle$, resp.~$\nabla^\lambda_{\tNC} \langle m \rangle$, appears in a standard, resp.~costandard, filtration of $X$. (This number does not depend on the filtration.) The general theory of graded highest weight categories implies that for any $\lambda \in \XB$ there exists a unique (up to isomorphism) indecomposable object $\TC^\lambda$ in $\Tilt(\ES^{\GB \times \Gm}(\tNC))$ which satisfies 
\[
(\TC^\lambda : \Delta_{\tNC}^\lambda)=1 \quad \text{and} \qquad (\TC^\lambda : \Delta^\mu_{\tNC} \langle m \rangle) \neq 0 \Rightarrow \mu \leq' \lambda.
\]
Moreover, every object in $\Tilt(\ES^{\GB \times \Gm}(\tNC))$ is a direct sum of objects of the form $\TC^\lambda \langle m \rangle$ for some $\lambda \in \XB$ and $m \in \ZM$. (See e.g.~\cite[Appendix~A]{modrap2} for references on this subject.)


The fact that $\ES^{\GB \times \Gm}(\tNC)$ is highest weight also has the following consequence (see~\cite[Lemma~A.5 \& Lemma~A.6]{modrap2}).

\begin{prop}
The natural functors
\[
\Kb \Tilt(\ES^{\GB \times \Gm}(\tNC)) \to \Db \ES^{\GB \times \Gm}(\tNC) \to D^{\GB \times \Gm}(\tNC)
\]
are equivalences of categories.\qed
\end{prop}

\subsection{Other consequences}
\label{ss:line-free}

The following results are consequences of Proposition~\ref{prop:exactness}. (See~\cite[Proposition~8.6]{arider2} for a different proof of some of these results, under some assumptions on $p$.)

\begin{prop}
\label{prop:objects-exotic}
\begin{enumerate}
\item
\label{it:line-bundles-exotic}
For $\lambda \in \XB$, the line bundle $\OC_{\tNC}(\lambda)$ belongs to the heart $\ES^{\GB \times \Gm}(\tNC)$.
\item
\label{it:free-exotic}
For any finite dimensional $\GB$-module $V$ and any $\lambda \in \XB$, the object $V \otimes \OC_{\tNC}(\lambda)$ belongs to $\ES^{\GB \times \Gm}(\tNC)$.
\item
\label{it:free-exact}
For any finite dimensional $\GB$-module $V$, the functor $V \otimes (-)$ is t-exact for the exotic t-structure.
\end{enumerate}
\end{prop}

\begin{proof}
\eqref{it:line-bundles-exotic}
Write $\lambda=\mu-\nu$ where $\mu$ and $\nu$ are dominant. Then 
\[
\OC_{\tNC}(\lambda) \cong \OC_{\tNC}(\mu) \otimes_{\OC_{\tNC}} \OC_{\tNC}(-\nu) \cong \JS_{\theta_{\mu}} \bigl( \OC_{\tNC}(-\nu) \bigr) \cong \JS_{T_{t_{\mu}}} \bigl( \OC_{\tNC}(-\nu) \bigr)
\]
since $\theta_\mu=t_\mu$.
Now $\OC_{\tNC}(-\nu)$ belongs to $\ES^{\GB \times \Gm}(\tNC)$ by
Proposition~\ref{prop:Delta}\eqref{it:Delta-antidom} and Corollary~\ref{cor:nabla-delta-heart}. By Proposition~\ref{prop:exactness}, this implies that $\OC_{\tNC}(\lambda)$ belongs to $\bigl( D^{\GB\times\Gm}(\tNC) \bigr)^{\leq 0}$.

On the other hand, we also have
\[
\OC_{\tNC}(\lambda) \cong \OC_{\tNC}(-\nu) \otimes_{\OC_{\tNC}} \OC_{\tNC}(\mu) \cong  \JS_{\theta_{-\nu}} \bigl( \OC_{\tNC}(\mu) \bigr) \cong \JS_{T_{t_\nu}^{-1}} \bigl( \OC_{\tNC}(\mu) \bigr).
\]
Now $\OC_{\tNC}(\mu)$ belongs to $\ES^{\GB \times \Gm}(\tNC)$ by
Corollaries~\ref{cor:nabla-dominant} and~\ref{cor:nabla-delta-heart}. Using Proposition~\ref{prop:exactness}, this implies that $\OC_{\tNC}(\lambda)$ belongs to $\bigl( D^{\GB\times\Gm}(\tNC) \bigr)^{\geq 0}$.

Finally we obtain that $\OC_{\tNC}(\lambda)$ belongs to $\bigl( D^{\GB\times\Gm}(\tNC) \bigr)^{\leq 0} \cap\bigl( D^{\GB\times\Gm}(\tNC) \bigr)^{\geq 0} = \ES^{\GB \times \Gm}(\tNC)$, proving the claim.

\eqref{it:free-exotic}
The coherent sheaf $V \otimes \OC_{\tNC}(\lambda)$ admits a filtration, in $\Coh^{\GB \times \Gm}(\tNC)$, whose subquotients are of the form $\OC_{\tNC}(\lambda+\mu)$ where $\mu$ is a $\TB$-weight of $V$. Hence the claim follows from~\eqref{it:line-bundles-exotic}.

\eqref{it:free-exact}
Let us first prove that the functor $V \otimes (-)$ is left-exact for any $V$. In fact it suffices to prove that if $\mu \in \XB$, then $V \otimes \nabla^\mu_{\tNC}$ belongs to $\bigl( D^{\GB\times\Gm}(\tNC) \bigr)^{\geq 0}$. However, if $v \in W$ is of minimal length such that $v \mu$ is dominant, then we have seen in the proof of Proposition~\ref{prop:nabla-delta-braid-group} that
\[
\nabla^\mu_{\tNC} \cong \JS_{(T_{v^{-1}})^{-1}} \bigl( \OC_{\tNC}(v\mu) \bigr).
\]
We deduce that $V \otimes \nabla^\mu_{\tNC} \cong \JS_{(T_{v^{-1}})^{-1}} \bigl( V \otimes \OC_{\tNC}(v\mu) \bigr)$. By~\eqref{it:free-exotic}, $V \otimes \OC_{\tNC}(v\mu)$ belongs to $\ES^{\GB \times \Gm}(\tNC)$. Then the claim follows from Proposition~\ref{prop:exactness}.

Next, we observe that the left-exactness of the functor $V^* \otimes (-)$ implies that its left-adjoint, namely $V \otimes (-)$, is right-exact; this finishes the proof.
\end{proof}

\begin{cor}
\label{cor:negative-cohomology}
If $\FC \in \ES^{\GB \times \Gm}(\tNC)$, then the complex of coherent sheaves $\FC$ is concentrated in non-negative degrees.
\end{cor}

\begin{proof}
By standard arguments, it suffices to prove that if $\lambda \in \XB^+$ is sufficiently large (in the sense that $\langle \lambda, {\check \alpha} \rangle$ is sufficiently big for all simple coroots ${\check \alpha}$), then the complex of $\GB \times \Gm$-modules $\Rder\Gamma(\tNC, \FC \otimes_{\OC_{\tNC}} \OC_{\tNC}(\lambda))$ is concentrated in non-negative degrees. In turn, given $\lambda$, to prove this property it suffices to prove that
\[
\Hom_{\Db \Rep(\GB \times \Gm)}^i(V, \Rder\Gamma(\tNC, \FC \otimes_{\OC_{\tNC}} \OC_{\tNC}(\lambda)))=0
\]
for any finite-dimensional $\GB \times \Gm$-module $V$ and any $i<0$. 
However we have
\begin{multline*}
\Hom_{\Db \Rep(\GB \times \Gm)}^i(V, \Rder\Gamma(\tNC, \FC \otimes_{\OC_{\tNC}} \OC_{\tNC}(\lambda))) \\
\cong \Hom^i_{D^{\GB \times \Gm}(\tNC)}(V \otimes \OC_{\tNC}(-\lambda), \FC).
\end{multline*}
Now $V \otimes \OC_{\tNC}(-\lambda)$ belongs to $\ES^{\GB \times \Gm}(\tNC)$ by Proposition~\ref{prop:objects-exotic}\eqref{it:free-exotic}, hence the claim holds by definition of a t-structure.
\end{proof}

\section{Tilting objects}
\label{sec:tilting}

\subsection{Geometric braid group action and (co)standard objets on the Gro\-thendieck resolution}
\label{ss:tgg}

In this section we will consider the Grothendieck resolution
\[
\tgg:= \GB \times^{\BB} (\gg/\ng)^*.
\]
This variety is a vector bundle over $\Flag$, and it is endowed with an action of $\GB \times \Gm$, where $\GB$ acts naturally and the action of $\Gm$ is induced by the action on $(\gg/\ng)^*$ where $t \in \Gm$ acts by multiplication by $t^{-2}$. 
Recall that there is a natural map $\nu \colon \tgg =  \GB \times^{\BB} (\gg/\ng)^* \to \tg^*$ induced by the $\BB$-equivariant restriction morphism $(\gg/\ng)^* \to \tg^*$. There exists also a natural morphism $\tgg \to \gg^*$ induced by the coadjoint action.

We will consider the categories
\[
D^{\GB \times \Gm}(\tgg):=\Db \Coh^{\GB \times \Gm}(\tgg), \quad D^{\GB}(\tgg):=\Db \Coh^{\GB}(\tgg).
\]
We denote by $\langle 1 \rangle$ the functor of tensoring with the one-dimensional tautological $\Gm$-module. Then, as in~\eqref{eqn:Hom-Gm-equiv}, for any $\FC$, $\GC$ in $D^{\GB \times \Gm}(\tgg)$, the forgetful functor induces an isomorphism
\begin{equation}
\label{eqn:Hom-Gm-equiv-tgg}
\bigoplus_{n \in \ZM} \Hom_{D^{\GB \times \Gm}(\tgg)}(\FC, \GC \langle n \rangle) \simto \Hom_{D^{\GB}(\tgg)}(\FC, \GC).
\end{equation}
For any $\lambda \in \XB$ we denote by $\OC_{\tgg}(\lambda)$ the pullback of $\OC_{\Flag}(\lambda)$ to $\tgg$.

The geometric braid group action considered in~\S\ref{ss:braid-action-Springer} can be ``extended'' to the category $D^{\GB \times \Gm}(\tgg)$ as follows. Let $s$ be a finite simple reflection, and let $\alpha$ be the associated simple root.
Recall the associated subscheme $Z_s \subset \tgg \times_{\gg^*} \tgg$ defined in~\cite[\S\S 1.3--1.5]{br}. (In fact, $Z_s$ is the closure of the graph of the action of $s$ on the regular part of $\tgg$.) Then we denote by
\[
\TM^s_{\tgg} \colon D^{\GB \times \Gm}(\tgg) \to D^{\GB \times \Gm}(\tgg), \quad \text{resp.} \quad
\SM^s_{\tgg} \colon D^{\GB \times \Gm}(\tgg) \to D^{\GB \times \Gm}(\tgg),
\]
the Fourier--Mukai transform with kernel $\OC_{Z_s} \langle -1 \rangle$, resp.~$\OC_{Z_s}(-\rho, \rho-\alpha) \langle -1 \rangle$. By \cite[Lemma~1.5.1, Proposition~1.10.3]{br}, these functors are quasi-inverse equivalences of categories.

By~\cite[Theorem~1.6.1]{br}, there exists a right action of the group $\Baff$ on the category $D^{\GB \times \Gm}(\tgg)$, where $T_s$ acts by the functor $\TM^s_{\tgg}$ (for any finite simple reflection $s$), and where $\theta_\lambda$ acts by tensoring with the line bundle $\OC_{\tgg}(\lambda)$ (for any $\lambda \in \XB$). (See~\S\ref{ss:braid-action-Springer} for the comparison with the action of~\cite{br}.) For $b \in \Baff$, we denote by
\[
\IS_b \colon D^{\GB \times \Gm}(\tgg) \simto D^{\GB \times \Gm}(\tgg)
\]
the action of $b$ (defined up to isomorphism).

For $s$ a finite simple reflection, associated with the simple root $\alpha$, we also set $\tgg_s := \GB \times^{\PB_s} (\gg / \pg_s^{\mathrm{nil}})^*$, where $\PB_s$ is the minimal standard parabolic subgroup associated with $s$, and $\pg_s^{\mathrm{nil}}$ is the Lie algebra of the unipotent radical of $\PB_s$. There exists a natural morphism $\tpi_s \colon \tgg \to \tgg_s$. 
By~\cite[Corollary~5.3.2]{riche}
there exists natural short exact sequences
\begin{gather}
\OC_{\Delta \tgg} \langle 2 \rangle \hookrightarrow \OC_{\tgg \times_{\tgg_s} \tgg} \twoheadrightarrow \OC_{Z_s}, \label{eqn:ses-tgg-1} \\
\OC_{Z_s}(\rho-\alpha,-\rho) \hookrightarrow \OC_{\tgg \times_{\tgg_s} \tgg} \twoheadrightarrow \OC_{\Delta \tgg}\label{eqn:ses-tgg-2}
\end{gather}
in $\Db \Coh^{\GB \times \Gm}(\tgg \times \tgg)$, where in each sequence the surjection is induced by restriction of functions, and where $\Delta \tgg \subset \tgg \times \tgg$ is the diagonal copy.

Let $i \colon \tNC \hookrightarrow \tgg$ be the natural embedding. Then we have isomorphisms
\begin{equation}
\label{eqn:action-i*}
\IS_b \circ i_* \cong i_* \circ \JS_b, \qquad \Lder i^* \circ \IS_b \cong \JS_b \circ \Lder i^*
\end{equation}
for any $b \in \BM_{\mathrm{aff}}$, see~\cite[Theorem~1.6.1]{br}. We have $\IS_{T_w}(\OC_{\tgg}) \cong \OC_{\tgg} \langle -\ell(w) \rangle$ for $w \in W$, see~\cite[Proof of Lemma~2.7.2]{br}.
Therefore,
if $\lambda \in \XB$ and if $w$ is any element in $W t_\lambda$, we can define the objects
\[
\nabla^\lambda_{\tgg} := \IS_{T_w}(\OC_{\tgg}) \langle -\ell(w_\lambda) + \ell(w) \rangle, \qquad
\Delta^\lambda_{\tgg} := \IS_{(T_{w^{-1}})^{-1}}(\OC_{\tgg}) \langle -\ell(w) + \ell(w_\lambda) \rangle
\]
in $D^{\GB \times \Gm}(\tgg)$. Then, by Proposition~\ref{prop:nabla-delta-braid-group} and~\eqref{eqn:action-i*}, there exist isomorphisms
\begin{equation}
\label{eqn:delta-nabla-i^*}
\Lder i^*(\nabla^\lambda_{\tgg}) \cong \nabla^\lambda_{\tNC}, \qquad \Lder i^*(\Delta^\lambda_{\tgg}) \cong \Delta^\lambda_{\tNC}.
\end{equation}

As in Remark~\ref{rmk:standard-action}, an object is of the form $\nabla^\lambda_{\tgg} \langle m \rangle$, resp.~$\Delta^\lambda_{\tgg} \langle m \rangle$, iff it is isomorphic to $\IS_{T_w}(\OC_{\tgg}) \langle p \rangle$, resp.~$\IS_{T_w^{-1}}(\OC_{\tgg}) \langle p \rangle$, for some $w \in \Waff$ and $p \in \ZM$.

\subsection{A ``Bott--Samelson type'' construction of tilting objects}
\label{ss:comparison}

If $s$ is a finite simple reflection,
we denote by 
\[
\Xi_s \colon D^{\GB \times \Gm}(\tgg) \to D^{\GB \times \Gm}(\tgg)
\]
the Fourier--Mukai transform associated with the kernel $\OC_{\tgg \times_{\tgg_s} \tgg} \langle -1 \rangle$. We can make a similar definition for $s_0$ an affine simple reflection: we choose (once and for all) a finite simple reflection $t$ and an element $b \in \BM_{\mathrm{aff}}$ such that $T_{s_0} = b T_t b^{-1}$ (see the proof of Proposition~\ref{prop:exactness} for existence), and set $\Xi_{s_0}:=\IS_{b^{-1}} \circ \Xi_t \circ \IS_b$. 

Using exact sequences~\eqref{eqn:ses-tgg-1} and~\eqref{eqn:ses-tgg-2}, for any simple reflection $s$, we obtain distinguished triangles of endofunctors of $D^{\GB \times \Gm}(\tgg)$:
\begin{gather}
\id \langle 1 \rangle \to \Xi_s \to \IS_{T_s} \xrightarrow{[1]}, \label{eqn:triangle-reflection-tgg-1} \\
\IS_{T_s^{-1}} \to \Xi_s \to \id \langle -1 \rangle \xrightarrow{[1]}. \label{eqn:triangle-reflection-tgg-2}
\end{gather}
(More precisely, if $s$ is finite then the distinguished triangles are obtained directly from~\eqref{eqn:ses-tgg-1} and~\eqref{eqn:ses-tgg-2}.
If $s$ is affine, with $T_s=b T_t b^{-1}$ as above, the distinguished triangles are obtained from those for $t$ by conjugation by $\IS_{b^{-1}}$.)

If $\mathsf{D}$ is any triangulated category, and if $X_1, \cdots, X_n$ are objects in $\mathsf{D}$, we will say that an object $X$ of $\mathsf{D}$ admits a ``filtration'' with ``subquotients'' $X_1, \cdots, X_n$ if $[X]$ belongs to $[X_1] * \cdots * [X_n]$. (Here, $[Y]$ is the isomorphism class of the object $Y$, and we use the ``$*$'' operation of~\cite[\S 1.3.9]{bbd}.)

\begin{lem}
\label{lem:Xi-standard-costandard-tgg}
If an object $\FC$ of $D^{\GB \times \Gm}(\tgg)$ admits a ``filtration'' with ``subquotients'' of the form $\nabla^\lambda_{\tgg} \langle m \rangle$ (resp.~$\Delta^\lambda_{\tgg} \langle m \rangle$) for $\lambda \in \XB$ and $m \in \ZM$, then so do the objects $\Xi_s(\FC)$ for any simple reflection $s \in \Waff$, and the objects $\IS_{T_\omega}(\FC)$ for any $\omega \in \Omega$.
\end{lem}

\begin{proof}
The case of $\IS_{T_\omega}$ is easy, and left to the reader. To prove the claim about the functor $\Xi_s$,
it suffices to prove that any object of the form $\Xi_s(\nabla^\lambda_{\tgg})$, resp.~ $\Xi_s(\Delta^\lambda_{\tgg})$, admits a ``filtration'' with ``subquotients'' of the form $\nabla^\mu_{\tgg} \langle m \rangle$ (resp.~$\Delta^\mu_{\tgg} \langle m \rangle$). We treat the case of $\nabla^\lambda_{\tgg}$; the case of $\Delta^\lambda_{\tgg}$ is similar.

By definition, there exists $w \in \Waff$ and $m \in \ZM$ such that $\nabla^\lambda_{\tgg} \cong \IS_{T_w}(\OC_{\tgg}) \langle m \rangle$. If $ws>w$, then $T_{ws}=T_w T_s$. Using triangle~\eqref{eqn:triangle-reflection-tgg-1} we obtain a distinguished triangle
\[
\nabla^\lambda_{\tgg} \langle 1 \rangle \to \Xi_s(\nabla^\lambda_{\tgg}) \to \IS_{T_{ws}}(\OC_{\tgg}) \langle m \rangle \xrightarrow{[1]}.
\]
Since the third term is isomorphic to $\nabla^\mu_{\tgg} \langle p \rangle$ for some $\mu \in \XB$ and $p \in \ZM$, this proves the claim in this case.

If $ws<w$ then $T_w T_s^{-1} = T_{ws}$, and using triangle~\eqref{eqn:triangle-reflection-tgg-2} we obtain a distinguished triangle
\[
\IS_{T_{ws}}(\OC_{\tgg}) \langle m \rangle \to \Xi_s(\nabla^\lambda_{\tgg}) \to \nabla^\lambda_{\tgg} \langle - 1 \rangle \xrightarrow{[1]}.
\]
As above the first term is isomorphic to $\nabla^\mu_{\tgg} \langle p \rangle$ for some $\mu \in \XB$ and $p \in \ZM$, hence the claim is proved in this case also.
\end{proof}


\begin{cor}
\label{cor:tilting-Bott-Samelson}
For any sequence $(s_1, \cdots, s_r)$ of simple reflections and for any $\omega \in \O$, the object
\[
\Lder i^* \bigl( \Xi_{s_1} \circ \cdots \circ \Xi_{s_r} \circ \IS_{T_\omega} (\OC_{\tgg}) \bigr)
\]
belongs to $\ES^{\GB \times \Gm}(\tNC)$, and is tilting therein. Moreover, any indecomposable tilting object $\TC^\lambda$ is a direct summand in an object of this form.
\end{cor}

\begin{proof}
The first claim is a direct consequence of Lemma~\ref{lem:Xi-standard-costandard-tgg}, using the fact that
if an object $\FC$ of $D^{\GB \times \Gm}(\tgg)$ admits a ``filtration'' with ``subquotients'' of the form $\nabla^\lambda_{\tgg} \langle m \rangle$ (resp.~$\Delta^\lambda_{\tgg} \langle m \rangle$), then $\Lder i^*(\FC)$ belongs to $\ES^{\GB \times \Gm}(\tNC)$ and admits a costandard filtration (resp.~a standard filtration), see~\eqref{eqn:delta-nabla-i^*}.

To prove the second claim choose a reduced expression $w_\lambda = \omega s_1 \cdots s_r$ (where $w_\lambda$ is as in~\S\ref{ss:exotic-t-structure}), and consider the object
\[
\TC := \Lder i^* \bigl( \Xi_{s_r} \circ \cdots \circ \Xi_{s_1} \circ \IS_{T_\omega} (\OC_{\tgg}) \bigr)
\]
Then $\TC$ is a tilting object in $\ES^{\GB \times \Gm}(\tNC)$, and it follows from Lemma~\ref{lem:Xi-standard-costandard-tgg} and its proof that $\nabla^\lambda_{\tNC}$ appears in a costandard filtration of $\TC$, and that $(\TC : \nabla^\mu_{\tNC} \langle m \rangle)=0$ unless $w_\mu \leq w_\lambda$, i.e.~unless $\mu \leq \lambda$. Hence $\TC^\lambda$ appears as a direct summand of $\TC$ (with multiplicity one).
\end{proof}

\begin{remark}
With the notation in the last paragraph of the proof of Corollary~\ref{cor:tilting-Bott-Samelson}, our arguments show that all the direct summands of the object $\TC$ which are different from $\TC^\lambda$ are of the form $\TC^\mu \langle m \rangle$ for $m \in \ZM$ and $\mu \in \XB$ which satisfies $\mu < \lambda$.
\end{remark}

\subsection{Dominant tilting objects}
\label{ss:tilting-exotic-dominant}

In this subsection we assume that $\GB$ is standard (i.e. that $p$ is a good prime for $\GB$ and $\gg$ admits a non-degenerate $\GB$-invariant bilinear form). In this case one can give an explicit description of the objects $\TC^\lambda$ for $\lambda \in \XB^+$ in terms of the indecomposable tilting $\GB$-modules $\Til(\lambda)$. This description is closely related to the description of tilting perverse coherent sheaves on $\NC$ obtained in~\cite{minn}.
Similar computations are also considered in~\cite[\S 4.2]{achar2}.

\begin{remark}
The assumption that $\GB$ is standard seems to be important in the results that appear below. (For instance, these results are used in~\cite{mr} in relation with the Mirkovi{\'c}--Vilonen conjecture, which is known to be false in bad characteristic.) In fact we expect that  without these assumptions one would have to replace the cotangent bundle $\tNC$ by the variety $\GB \times^\BB \UB$ as in~\cite{klt}. But there is no obvious counterpart of our $\Gm$-action on this variety, which creates complications that we do not consider here.
\end{remark}

We need to introduce more notation. 
For $\lambda \in \XB^+$, we set
\[
\St(\lambda) := \bigl( \mathsf{Ind}_\BB^\GB(-w_0 \lambda) \bigr)^*, \qquad \Cost(\lambda) := \mathsf{Ind}_\BB^\GB(\lambda).
\]
Then the category of finite dimensional $\GB$-modules is a highest weight category, with standard, resp.~costandard, objects $\St(\lambda)$, resp.~$\Cost(\lambda)$, for $\lambda \in \XB^+$. Recall that we say that a finite dimensional $\GB$-module $V$ \emph{admits a good filtration} if it admits a filtration with subquotients of the form $\Cost(\nu)$ where $\nu \in \XB^+$. In this case the number of times $\Cost(\nu)$ appears in such a filtration does not depend on the choice of filtration, and is denoted $(V : \Cost(\nu))$. Dually, we say that $V$ \emph{admits a Weyl filtration} if it admits a filtration with subquotients $\St(\nu)$ where $\nu \in \XB^+$; in this case we denote by $(V : \St(\nu))$ the corresponding multiplicity.
For $\lambda \in \XB^+$, we denote by $\Til(\lambda)$ the indecomposable tilting $\GB$-module with highest weight $\lambda$.

We will also use ``Lusztig's $q$-analogue,'' which is defined as follows (see~\cite{lusztig}). First we define Kostant's partition function via the equality
\[
\prod_{\alpha \in \Phi_+} \frac{1}{1-\vv e^{\alpha}} = \sum_{\lambda \in \XB} \PC_\lambda(\vv) e^\lambda
\]
in a suitable completion of the group ring of $\XB$ over $\ZM[\vv,\vv^{-1}]$ (where $\vv$ is an indeterminate). Then we set
\[
\MC_\lambda^\mu(\vv) = \sum_{w \in W} (-1)^{\ell(w)} \PC_{w(\lambda+\rho)-(\mu+\rho)}(\vv).
\]
Note that we have, for $\lambda, \mu \in \XB^+$,
\begin{equation}
\label{eqn:q-analogue}
\MC_\lambda^\mu(\vv) \neq 0 \quad \Rightarrow \quad \bigl( \mu \preceq \lambda \bigr).
\end{equation}
Finally, if $\mu \in \XB$, we denote by $\dom(\mu)$ the unique dominant weight in the $W$-orbit of $\mu$.

The following theorem collects some well-known (though difficult) results. Here, for a $\GB \times \Gm$-module $V$, we denote by $V_k$ the sub-$\GB$-module of $V$ on which $\Gm$ acts via $z \mapsto z^k$.

\begin{thm}
\label{thm:line-bundles-tNC}
Assume $\GB$ is standard, and let $\lambda \in \XB^+$.
\begin{enumerate}
\item
\label{it:vanishing}
We have 
\[
\mathsf{H}^i(\tNC, \OC_{\tNC}(\lambda))=0 \quad \text{for $i>0$.}
\]
\item
\label{it:character}
For any $k \in \ZM$,
the $\GB$-module $\Gamma(\tNC, \OC_{\tNC}(\lambda))_k$ admits a good filtration. Moreover, for $\nu \in \XB^+$ we have
\[
\sum_{k \in \ZM} \bigl( \Gamma(\tNC, \OC_{\tNC}(\lambda))_k : \Cost(\nu) \bigr) \cdot \vv^k = \MC_\nu^\lambda(\vv^2).
\]
\end{enumerate}
\end{thm}

\begin{proof}
Statement~\eqref{it:vanishing} is proved in~\cite[Theorem~2]{klt}.
The first statement in~\eqref{it:character} is proved in~\cite[Theorem~7]{klt}.
And then the second statement follows from~\cite[Lemma~6.1]{brylinski}. (In~\cite{brylinski} it is assumed that the base field has characteristic $0$. However, the proof also applies in positive characteristic, replacing the reference to the Borel--Weil--Bott theorem by a reference to~\cite[\S II.5.9, Equation (1)]{jantzen}).
\end{proof}

In the following statements we use the fact that for any finite dimensional $\GB$-module $V$, the object $V \otimes \OC_{\tNC}$ belongs to $\EC^{\GB \times \Gm}(\tNC)$, see Proposition~\ref{prop:objects-exotic}\eqref{it:free-exotic}.

\begin{prop}
\label{prop:Weyl-standard}
Assume $\GB$ is standard, and
let $V$ be a finite dimensional $\GB$-module which admits a Weyl filtration. Then the object $V \otimes \OC_{\tNC} \in \EC^{\GB \times \Gm}(\tNC)$ admits a standard filtration. Moreover, for any $\mu \in \XB$ we have
\[
\sum_{i \in \ZM} (V \otimes \OC_{\tNC} : \Delta_{\tNC}^\mu \langle i \rangle) \cdot \vv^i= \vv^{-\delta(\mu)} \cdot \sum_{\nu \in \XB^+} (V : \St( \nu)) \cdot \MC^{\dom(\mu)}_\nu(\vv^{-2}).
\]
\end{prop}

\begin{proof}
Let $\mu \in \XB$, and let $v \in W$ be of minimal length such that $v \mu=\dom(\mu)$. Then we have seen in the proof of Proposition~\ref{prop:nabla-delta-braid-group} that
\[
\nabla^\mu_{\tNC} \cong \JS_{(T_{v^{-1}})^{-1}} \bigl( \OC_{\tNC}(v\mu) \bigr).
\]
We deduce that for $i,j \in \ZM$ we have
\begin{multline*}
\Hom^i_{D^{\GB \times \Gm}(\tNC)}(V \otimes \OC_{\tNC}, \nabla^\mu_{\tNC} \langle j \rangle) \cong \\
\Hom^i_{D^{\GB \times \Gm}(\tNC)} \Bigl(V \otimes \OC_{\tNC},  \JS_{(T_{v^{-1}})^{-1}} \bigl( \OC_{\tNC}(v\mu) \bigr) \langle j \rangle \Bigr) \cong \\
\Hom^i_{D^{\GB \times \Gm}(\tNC)} \Bigl( \JS_{T_{v^{-1}}} \bigl( V \otimes \OC_{\tNC} \bigr), \OC_{\tNC}(v\mu) \langle j \rangle \Bigr) \cong \\
\Hom^i_{D^{\GB \times \Gm}(\tNC)} \bigl( V \otimes \OC_{\tNC}, \OC_{\tNC}(v\mu) \langle j+\ell(v) \rangle \bigr),
\end{multline*}
where the last isomorphism uses~\eqref{eqn:J-O}. Now we observe that
\begin{multline*}
\Hom^i_{D^{\GB \times \Gm}(\tNC)} \bigl( V \otimes \OC_{\tNC}, \OC_{\tNC}(v\mu) \langle j+\ell(v) \rangle \bigr) \cong \\
\Hom^i_{\Db \Rep(\GB \times \Gm)} \bigl(V, \Rder\Gamma(\tNC, \OC_{\tNC}(v\mu)) \langle j+\ell(v) \rangle \bigr).
\end{multline*}
Using Theorem~\ref{thm:line-bundles-tNC}\eqref{it:vanishing}, we finally obtain that
\begin{equation}
\label{eqn:Hom-tilt-nabla}
\begin{split}
& \Hom^i_{D^{\GB \times \Gm}(\tNC)}(V \otimes \OC_{\tNC}, \nabla^\mu_{\tNC} \langle j \rangle) \cong \\
& \qquad \qquad \qquad \qquad \Ext^i_{\Rep(\GB)} \bigl( V, \Gamma(\tNC, \OC_{\tNC}(\dom(\mu)))_{-j-\delta(\mu)} \bigr).
\end{split}
\end{equation}
Since $V$ admits a Weyl filtration and $\Gamma(\tNC, \OC_{\tNC}(\dom(\mu)))_{-j-\delta(\mu)}$ admits a good filtration (see Theorem~\ref{thm:line-bundles-tNC}\eqref{it:character}), this implies that
\[
\Hom^i_{D^{\GB \times \Gm}(\tNC)}(V \otimes \OC_{\tNC}, \nabla^\mu_{\tNC} \langle j \rangle)=0 \qquad \text{for $i>0$.}
\]
By the general theory of graded highest weight categories, this condition implies that $V \otimes \OC_{\tNC}$ admits a standard filtration.

To compute the multiplicities, we observe that
\[
 \sum_{i \in \ZM} (V \otimes \OC_{\tNC} : \Delta^\mu_{\tNC} \langle i \rangle) \cdot \vv^i
= \sum_{i \in \ZM} \dim \bigl( \Hom(V \otimes \OC_{\tNC}, \nabla^\mu_{\tNC} \langle i \rangle) \bigr) \cdot \vv^i .
\]
Using~\eqref{eqn:Hom-tilt-nabla}, we deduce that this sum equals
\[
\sum_{\substack{\nu \in \XB^+ \\ i \in \ZM}} (V : \St(\nu)) \cdot (\Gamma(\tNC, \OC_{\tNC}(\dom(\mu)))_{-i-\delta(\mu)} : \Cost(\nu)) \cdot \vv^i.
\]
Then the formula follows from Theorem~\ref{thm:line-bundles-tNC}\eqref{it:character}.
\end{proof}

We also have a ``dual'' statement, as follows.

\begin{prop}
\label{prop:good-costandard}
Assume $\GB$ is standard, and
let $V$ be a finite dimensional $\GB$-module which admits a good filtration. Then the object $V \otimes \OC_{\tNC} \in \EC^{\GB \times \Gm}(\tNC)$ admits a costandard filtration. Moreover, for any $\mu \in \XB$ we have
\[
\sum_{i \in \ZM} (V \otimes \OC_{\tNC} : \nabla_{\tNC}^\mu \langle i \rangle) \cdot \vv^i = \vv^{\delta(\mu)} \cdot \sum_{\nu \in \XB^+} (V : \Cost(-w_0 \nu)) \cdot \MC^{\dom(-\mu)}_\nu(\vv^{2}).
\]
\end{prop}

\begin{proof}
The proof is similar to that of Proposition~\ref{prop:Weyl-standard}. In fact, one can check that for $\mu \in \XB$ we have
\begin{equation}
\label{eqn:Hom-tilt-delta}
\begin{split}
& \Hom^i_{D^{\GB \times \Gm}(\tNC)}(\Delta^\mu_{\tNC} \langle j \rangle, V \otimes \OC_{\tNC}) \cong \\
& \qquad \qquad \qquad \qquad \Ext^i_{\Rep(\GB)} \bigl( V^*, \Gamma(\tNC, \OC_{\tNC}(\dom(-\mu)))_{j-\delta(\mu)} \bigr).
\end{split}
\end{equation}
Since $V^*$ admits a Weyl filtration and $\Gamma(\tNC, \OC_{\tNC}(\dom(-\mu)))_{j-\delta(\mu)}$ admits a good filtration (see Theorem~\ref{thm:line-bundles-tNC}\eqref{it:character}), it follows that
\[
\Hom^i_{D^{\GB \times \Gm}(\tNC)}(\Delta^\mu_{\tNC} \langle j \rangle, V \otimes \OC_{\tNC})=0 \quad \text{for $i > 0$},
\]
which implies that $V \otimes \OC_{\tNC}$ admits a costandard filtration.

The formula for the multiplicities follows from~\eqref{eqn:Hom-tilt-delta} and Theorem~\ref{thm:line-bundles-tNC}\eqref{it:character}.
\end{proof}

The following easy corollary turns out to be a key ingredient of our proof that spherical parity sheaves on affine Grassmannians are perverse in~\cite{mr}.

\begin{cor}
\label{cor:tilting-exotic-dominant}
Assume $\GB$ is standard. Then for all $\lambda \in \XB^+$ we have 
\[
\TC^\lambda \cong \Til(\lambda) \otimes \OC_{\tNC}.
\] 
\end{cor}

\begin{proof}
By Propositions~\ref{prop:Weyl-standard} and~\ref{prop:good-costandard}, the object $\Til(\lambda) \otimes \OC_{\tNC}$ is tilting.
Next, let us prove that this object is indecomposable in $\ES^{\GB \times \Gm}(\tNC)$. For this, it suffices to remark that
\begin{multline*}
\Hom_{D^{\GB \times \Gm}(\tNC)}(\Til(\lambda) \otimes \OC_{\tNC}, \Til(\lambda) \otimes \OC_{\tNC}) \cong \\
\Hom_{\Db \Rep(\GB \times \Gm)} \bigl( \Til(\lambda), \Til(\lambda) \otimes \Rder\Gamma(\tNC, \OC_{\tNC}) \bigr) \cong \\
\Hom_{\Db \Rep(\GB \times \Gm)} \bigl( \Til(\lambda),  \Til(\lambda) \otimes \Gamma(\tNC, \OC_{\tNC}) \bigr) \cong \\
\Hom_{\Rep(\GB)}(\Til(\lambda), \Til(\lambda)).
\end{multline*}
Here the second isomorphism uses Theorem~\ref{thm:line-bundles-tNC}\eqref{it:vanishing}, and the third one the observation that $(\Gamma(\tNC, \OC_{\tNC}))^{\Gm}=\FM$, see e.g.~the argument in the proof of~\cite[Lemma~1.4.2]{br}. Since $\Til(\lambda)$ is indecomposable, the algebra $\Hom_{\Rep(\GB)}(\Til(\lambda), \Til(\lambda))$ is local, which implies the indecomposability of $\Til(\lambda) \otimes \OC_{\tNC}$.

Using the formula in Proposition~\ref{prop:Weyl-standard} and~\eqref{eqn:q-analogue}, we observe that if $(\Til(\lambda) \otimes \OC_{\tNC} : \Delta^\mu_{\tNC} \langle j \rangle) \neq 0$, then there exists $\nu \in \XB^+$ such that $\dom(\mu) \preceq \nu$ and $\nu \preceq \lambda$, which implies that $\dom(\mu) \preceq \lambda$. By basic properties of the order $\leq$ (see~\S\ref{ss:exotic-t-structure}), this implies that $\mu \leq \lambda$. This formula also implies that $(\Til(\lambda) \otimes \OC_{\tNC} : \Delta^\lambda_{\tNC})=1$. Hence 
$\Til(\lambda) \otimes \OC_{\tNC}$ satisfies the properties that characterize $\TC^\lambda$, which finishes the proof.
\end{proof}

\begin{remark}
In case $p=0$, Corollary~\ref{cor:tilting-exotic-dominant} can also be deduced from~\cite[Lemma 12]{bezru-tilting} and~\cite[Theorem~7]{arkb}.
\end{remark}

\subsection{Tensoring with a tilting $\GB$-module}

In this subsection we assume again that $\GB$ is standard. The goal of this subsection is to prove the following result, which is a generalization of part of Corollary~\ref{cor:tilting-exotic-dominant}.

\begin{prop}
\label{prop:tensor-tilting}
Assume $\GB$ is standard. Then for any finite dimensional tilting $\GB$-module $V$ and any $\TC \in \Tilt(\ES^{\GB \times \Gm}(\tNC))$, the object $V \otimes \TC$ belongs to $\Tilt(\ES^{\GB \times \Gm}(\tNC))$.
\end{prop}

Before giving the proof of this proposition, we need some preparation. First we observe that since the morphism $\nu \colon \tgg \to \tg^*$ is $\GB$-equivariant (where $\GB$ acts trivially on $\tg^*$), for $\FC,\GC$ in $\QCoh^\GB(\tgg)$ the vector space $\Hom_{\QCoh^\GB(\tgg)}(\FC,\GC)$ has a natural structure of $\OC(\tg^*)$-module. Therefore, the bifunctor $\Rder\Hom_{\QCoh^\GB(\tgg)}(-,-)$ factors through a bifunctor with values in $D \Mod(\OC(\tg^*))$, which we denote similarly.

\begin{lem}
\label{lem:Hom-tNC-tgg}
For any $\FC, \GC$ in $D^{\GB}(\tgg)$, there exists a canonical isomorphism
\[
\FM \lotimes_{\OC(\tg^*)} \Rder\Hom_{\Coh^{\GB}(\tgg)}(\FC,\GC) \simto \Rder\Hom_{\Coh^{\GB}(\tNC)}(\Lder i^* \FC, \Lder i^* \GC).
\]
\end{lem}

\begin{proof}
Consider the flat, affine group scheme $\GB \times \tg^*$ over $\OC(\tg^*)$. Then by Proposition~\ref{prop:Hom-equivariant} there exists a canonical isomorphism
\[
\Rder\Hom_{\Coh^{\GB}(\tgg)}(\FC,\GC) \simto \Inv^{\GB \times \tg^*} \Rder\Hom_{\Coh(\tgg)}(\FC,\GC)
\]
in $\Db \Mod(\OC(\tg^*))$. (See Appendix~\ref{sec:appendix} for the definition of the functor $\Inv^{\GB \times \tg^*}$.)
%
By Proposition~\ref{prop:Inv-tensor},
we also have a canonical isomorphism of functors
\begin{equation}
\label{eqn:Inv-tensor}
\FM \lotimes_{\OC(\tg^*)} \Inv^{\GB \times \tg^*}(-) \simto \Inv^\GB( \FM \lotimes_{\OC(\tg^*)}(-)).
\end{equation}

Then our claim is obtained from the following chain of isomorphisms, where we identify quasi-coherent sheaves on $\tg^*$ and $\OC(\tg^*)$-modules in the standard way:
\begin{align*}
\FM \lotimes_{\OC(\tg^*)} \Rder\Hom_{\Coh^\GB(\tgg)}(\FC,\GC) & \cong \FM \lotimes_{\OC(\tg^*)} \Inv^{\GB \times \tg^*} \bigl( \nu_* \Rder\sheafHom_{\tgg}(\FC,\GC) \bigr) \\
&\cong \Inv^{\GB} \bigl( \FM \lotimes_{\OC(\tg^*)} \nu_* \Rder\sheafHom_{\tgg}(\FC,\GC) \bigr) \\
&\cong \Inv^{\GB} \bigl( \Rder\Gamma(\tNC, \Lder i^* \Rder\sheafHom_{\tgg}(\FC,\GC)) \bigr) \\
&\cong \Inv^{\GB} \bigl( \Rder\Gamma(\tNC, \Rder\sheafHom_{\tNC}(\Lder i^* \FC, \Lder i^* \GC)) \bigr) \\
&\cong \Rder\Hom_{\Coh^{\GB}(\tNC)}(\Lder i^* \FC, \Lder i^* \GC).
\end{align*}
(Here the second isomorphism follows from~\eqref{eqn:Inv-tensor}, the third one from the base change theorem -- in the form of~\cite[Theorem~3.10.3]{lipman} -- and the last one from Proposition~\ref{prop:Hom-equivariant} again. The other isomorphisms are obvious.)
\end{proof}

\begin{lem}
\label{lem:global-coh-nabla-delta}
Assume $\GB$ is standard. 
\begin{enumerate}
\item
\label{it:global-coh-nabla}
For any $\mu \in \XB$, we have
\[
\Hom^n_{\Db \Coh(\tgg)}(\OC_{\tgg}, \nabla^\mu_{\tgg})=0 \qquad \text{for $n \neq 0$.}
\]
Moreover, $\Hom_{\Db \Coh(\tgg)}(\OC_{\tgg}, \nabla^\mu_{\tgg})$ admits a $\ZM_{\geq 0}$-filtration, as a $\GB \times \Gm$-equivariant $\OC(\tg^*)$-module, with associated graded 
\[
\OC(\tg^*) \otimes \Gamma(\tNC, \OC_{\tNC}(\dom(\mu))) \langle \delta(\mu) \rangle.
\]
\item
\label{it:global-coh-delta}
For any $\mu \in \XB$, we have
\[
\Hom^n_{\Db \Coh(\tgg)}(\Delta^\mu_{\tgg}, \OC_{\tgg})=0 \qquad \text{for $n \neq 0$.}
\]
Moreover, $\Hom_{\Db \Coh(\tgg)}(\Delta^\mu_{\tgg}, \OC_{\tgg})$ admits a $\ZM_{\geq 0}$-filtration, as a $\GB \times \Gm$-equivariant $\OC(\tg^*)$-module, with associated graded 
\[
\OC(\tg^*) \otimes \Gamma(\tNC, \OC_{\tNC}(\dom(-\mu))) \langle \delta(\mu) \rangle.
\]
\end{enumerate}
\end{lem}

\begin{proof}
We prove~\eqref{it:global-coh-nabla}; the proof of~\eqref{it:global-coh-delta} is similar.

Let
$v \in W$ be of minimal length such that $v \mu$ is dominant. Then, as in the proof of Proposition~\ref{prop:nabla-delta-braid-group}, we have
\[
\nabla^\mu_{\tgg} \cong \IS_{(T_{v^{-1}})^{-1}} \bigl( \OC_{\tgg}(v\mu) \bigr).
\]
It follows that
\begin{multline*}
\Hom^n_{\Db \Coh(\tgg)}(\OC_{\tgg}, \nabla^\mu_{\tgg}) \cong \Hom^n_{\Db \Coh(\tgg)} \bigl( \OC_{\tgg}, \IS_{(T_{v^{-1}})^{-1}} ( \OC_{\tgg}(v\mu) ) \bigr) \\
\cong \Hom^n_{\Db \Coh(\tgg)} ( \OC_{\tgg}, \OC_{\tgg}(v\mu) ) \langle \ell(v) \rangle \cong \mathsf{H}^n(\tgg, \OC_{\tgg}(v\mu) ) \langle \ell(v) \rangle
\end{multline*}
since $\IS_{T_{v^{-1}}} ( \OC_{\tgg}) \cong \OC_{\tgg} \langle - \ell(v) \rangle$, see~\S\ref{ss:tgg}. Hence what we have to prove is that if $\lambda \in \XB^+$, then $\mathsf{H}^i(\tgg, \OC_{\tgg}(\lambda))=0$ for $i > 0$, and moreover $\Gamma(\tgg, \OC_{\tgg}(\lambda))$ admits a filtration (as a $\GB \times \Gm$-equivariant $\OC(\tg^*)$-module) with associated graded $\OC(\tg^*) \otimes \Gamma(\tNC, \OC_{\tNC}(\lambda))$. 

This proof is standard: we have natural isomorphisms
\[
\mathsf{H}^n(\tgg, \OC_{\tgg}(\lambda)) \cong \Rder^n \mathsf{Ind}_\BB^\GB(\mathrm{S}(\gg/\ng) \otimes \FM_\lambda), \qquad \mathsf{H}^n(\tNC, \OC_{\tNC}(\lambda)) \cong \Rder^n \mathsf{Ind}_\BB^\GB(\mathrm{S}(\gg/\bg) \otimes \FM_\lambda)
\]
where $\mathrm{S}(\gg/\ng)$, resp.~$\mathrm{S}(\gg/\bg)$, is the symmetric algebra of the $\BB$-module $\gg/\ng$, resp.~$\gg/\bg$. Now there exists a natural exact sequence of $\BB$-modules $\tg \hookrightarrow \gg/\ng \twoheadrightarrow \gg/\bg$ (where $\BB$ acts trivially on $\tg$), hence the $\BB$-equivariant $\OC(\tg^*)$-module $\mathrm{S}(\gg/\ng)$ admits a filtration with associated graded $\OC(\tg^*) \otimes \mathrm{S}(\gg/\bg)$. Using the fact that $\mathsf{H}^n(\tNC, \OC_{\tNC}(\lambda))=0$ for $n \neq 0$ (see Theorem~\ref{thm:line-bundles-tNC}\eqref{it:vanishing}) and that the functor $\Rder^n \mathsf{Ind}_\BB^\GB(-)$ commutes with direct sums (see~\cite[Remark in~\S I.4.17]{jantzen}), this implies our claim.
\end{proof}

\begin{cor}
\label{cor:Hom-BS-delta-nabla}
Assume $\GB$ is standard. 
Let $\omega \in \Omega$ and let $(s_1, \cdots, s_r)$ be a sequence of simple reflections.
\begin{enumerate}
\item
\label{it:Hom-BS-nabla}
We have
\[
\Hom^n_{\Db\Coh(\tgg)} \bigl( \Xi_{s_1} \circ \cdots \circ \Xi_{s_r} \circ \IS_{T_\omega} (\OC_{\tgg}), \nabla^\mu_{\tgg} \bigr) =0 \quad \text{for $n \neq 0$.}
\]
Moreover, $\Hom_{\Db\Coh(\tgg)} \bigl( \Xi_{s_1} \circ \cdots \circ \Xi_{s_r} \circ \IS_{T_\omega} (\OC_{\tgg}), \nabla^\mu_{\tgg} \bigr)$ admits a $\ZM_{\geq 0}$-filtration whose subquotients have the form $\OC(\tg^*) \otimes V \langle m \rangle$ where $m \in \ZM$ and $V$ is a $\GB$-module which admits a good filtration.
\item
\label{it:Hom-Delta-BS}
We have
\[
\Hom^n_{\Db\Coh(\tgg)} \bigl( \Delta^\mu_{\tgg}, \Xi_{s_1} \circ \cdots \circ \Xi_{s_r} \circ \IS_{T_\omega} (\OC_{\tgg}) \bigr)=0 \quad \text{for $n \neq 0$.}
\]
Moreover, $\Hom_{\Db\Coh(\tgg)} \bigl( \Delta^\mu_{\tgg}, \Xi_{s_1} \circ \cdots \circ \Xi_{s_r} \circ \IS_{T_\omega} (\OC_{\tgg}) \bigr)$ admits a $\ZM_{\geq 0}$-filtration whose subquotients have the form $\OC(\tg^*) \otimes V \langle m \rangle$ where $m \in \ZM$ and $V$ is a $\GB$-module which admits a good filtration.
\end{enumerate}
\end{cor}

\begin{proof}
We prove~\eqref{it:Hom-BS-nabla}; the proof of~\eqref{it:Hom-Delta-BS} is similar.

First we claim that for any simple reflection $s$ and any $\FC,\GC$ in $D^{\GB \times \Gm}(\tgg)$ there exists a canonical isomorphism of $\GB \times \Gm$-modules
\begin{equation}
\label{eqn:Hom-Xi}
\Hom_{\Db \Coh(\tgg)} \bigl( \Xi_s(\FC), \GC \bigr) \cong \Hom_{\Db \Coh(\tgg)} \bigl( \FC, \Xi_s(\GC) \bigr) \langle -2 \rangle.
\end{equation}
Indeed, by definition of $\Xi_s$ one can assume that $s$ is a \emph{finite} simple reflection. Then if we still denote by $\Xi_s$ the endofunctor of $\Db \Coh(\tgg)$ defined as in~\S\ref{ss:comparison}, by~\cite[Proposition~5.2.2]{riche}, there exists a canonical isomorphism of functors
$\Xi_s \cong L(\tpi_s)^* \circ \Rder(\tpi_s)_*$,
where $\tpi_s \colon \tgg \to \tgg_s$ is the morphism considered in~\S\ref{ss:tgg}. Hence, since $\tpi_s$ is a projective morphism, we have a canonical isomorphism
\[
\Hom_{\Db \Coh(\tgg)} \bigl( \Xi_s(\FC), \GC \bigr) \cong \Hom_{\Db \Coh(\tgg)} \bigl( \FC, (\tpi_s)^! \circ \Rder(\tpi_s)_*(\GC) \bigr).
\]
Since $\tgg$ and $\tgg_s$ are smooth of dimension $\dim(\gg)$, we deduce a canonical isomorphism
\[
\Hom_{\Db \Coh(\tgg)} \bigl( \Xi_s(\FC), \GC \bigr) \cong \Hom_{\Db \Coh(\tgg)} \bigl( \FC, \omega_{\tgg} \otimes_{\OC_{\tgg}} L(\tpi_s)^* (\omega_{\tgg_s}^{-1} \otimes_{\OC_{\tgg_s}} \Rder(\tpi_s)_*(\GC)) \bigr).
\]
This isomorphism is $\GB \times \Gm$-equivariant by functoriality of the constructions in the duality theorem~\cite[Theorem~VII.3.3]{hartshorne}.
Finally we observe that we have $\omega_{\tgg} \cong \OC_{\tgg} \langle 2\dim(\gg/\bg) \rangle$ and $\omega_{\tgg_s} \cong \OC_{\tgg} \langle  2\dim(\gg/\pg_s^{\mathrm{nil}})\rangle$ as $\GB \times \Gm$-equivariant coherent sheaves, and the proof of~\eqref{eqn:Hom-Xi} is complete.

Using~\eqref{eqn:Hom-Xi}, we obtain a canonical isomorphism
\begin{multline*}
\Hom^n_{\Db\Coh(\tgg)} \bigl( \Xi_{s_1} \circ \cdots \circ \Xi_{s_r} \circ \IS_{T_\omega} (\OC_{\tgg}), \nabla^\mu_{\tgg} \bigr) \cong \\
\Hom^n_{\Db\Coh(\tgg)} \bigl( \OC_{\tgg}, \IS_{T_{\omega^{-1}}} \circ \Xi_{s_r} \circ \cdots \circ \Xi_{s_1}(\nabla^\mu_{\tgg}) \bigr) \langle -2r \rangle.
\end{multline*}
Now, by Lemma~\ref{lem:Xi-standard-costandard-tgg}, the object $\IS_{T_{\omega^{-1}}} \circ \Xi_{s_r} \circ \cdots \circ \Xi_{s_1} (\nabla^\mu_{\tgg})$ admits (in $D^{\GB \times \Gm}(\tgg)$) a ``filtration'' with ``subquotients'' of the form $\nabla^\nu_{\tgg} \langle m \rangle$; hence the desired properties follow from Lemma~\ref{lem:global-coh-nabla-delta}\eqref{it:global-coh-nabla} and Theorem~\ref{thm:line-bundles-tNC}\eqref{it:character}.
\end{proof}

\begin{proof}[Proof of Proposition~{\rm \ref{prop:tensor-tilting}}]
By Corollary~\ref{cor:tilting-Bott-Samelson}, it suffices to prove that for any $\omega \in \Omega$ and any sequence $(s_1, \cdots, s_r)$ of simple reflections, the object
\[
V \otimes
\Lder i^* \bigl( \Xi_{s_1} \circ \cdots \circ \Xi_{s_r} \circ \IS_{T_\omega} (\OC_{\tgg}) \bigr) \cong \Lder i^* \bigl( \Xi_{s_1} \circ \cdots \circ \Xi_{s_r} \circ \IS_{T_\omega} (V \otimes \OC_{\tgg}) \bigr)
\]
is tilting. And for this it suffices to prove that for all $\mu \in \XB$ we have
\begin{align*}
\Hom_{D^{\GB}(\tNC)}^n \bigl( \Lder i^*( \Xi_{s_1} \circ \cdots \circ \Xi_{s_r} \circ \IS_{T_\omega} (V \otimes \OC_{\tgg}))
, \nabla^\mu_{\tNC} \bigr)=0 \qquad \text{for $n \neq 0$,} \\
\Hom_{D^{\GB}(\tNC)}^n \bigl( \Delta^\mu_{\tNC}, \Lder i^* ( \Xi_{s_1} \circ \cdots \circ \Xi_{s_r} \circ \IS_{T_\omega} (V \otimes \OC_{\tgg}) ) \bigr)=0 \qquad \text{for $n \neq 0$.}
\end{align*}
We explain the proof of the first condition; the proof of the second one is similar.

Using Lemma~\ref{lem:Hom-tNC-tgg} and~\eqref{eqn:delta-nabla-i^*}, it suffices to prove that the complex of $\OC(\tg^*)$-modules
\[
\Rder\Hom_{\Coh^{\GB}(\tgg)} \bigl( \Xi_{s_1} \circ \cdots \circ \Xi_{s_r} \circ \IS_{T_\omega} (V \otimes \OC_{\tgg}), \nabla^\mu_{\tgg} \bigr)
\]
is concentrated in degree $0$, and free. Now by Proposition~\ref{prop:Hom-equivariant} and Corollary~\ref{cor:Hom-BS-delta-nabla}\eqref{it:Hom-BS-nabla} this complex is isomorphic to
\begin{multline*}
\Inv^{\GB} \bigl( \Rder\Hom_{\Coh(\tgg)}(\Xi_{s_1} \circ \cdots \circ \Xi_{s_r} \circ \IS_{T_\omega} (V \otimes \OC_{\tgg}), \nabla^\mu_{\tgg}) \bigr) \cong \\
\Inv^{\GB} \bigl( V^* \otimes \Rder\Hom_{\Coh(\tgg)}(\Xi_{s_1} \circ \cdots \circ \Xi_{s_r} \circ \IS_{T_\omega} (\OC_{\tgg}), \nabla^\mu_{\tgg}) \bigr) \cong \\
\Inv^{\GB} \bigl( V^* \otimes \Hom_{\Coh(\tgg)}(\Xi_{s_1} \circ \cdots \circ \Xi_{s_r} \circ \IS_{T_\omega} (\OC_{\tgg}), \nabla^\mu_{\tgg}) \bigr) \cong \\
\Rder\Hom_{\Rep(\GB)} \bigl( V, \Hom_{\Coh(\tgg)}(\Xi_{s_1} \circ \cdots \circ \Xi_{s_r} \circ \IS_{T_\omega} (\OC_{\tgg}), \nabla^\mu_{\tgg}) \bigr).
\end{multline*}
Using again Corollary~\ref{cor:Hom-BS-delta-nabla}\eqref{it:Hom-BS-nabla}, the fact that $V$ admits a Weyl filtration, and the fact that the functor $\Hom^n_{\Rep(\GB)}(V,-)$ commutes with direct sums (see~\cite[Lem\-ma~I.4.7]{jantzen}), we obtain that
\[
\Hom^n_{\Db \Rep(\GB)} \bigl( V, \Hom_{\Coh(\tgg)}(\Xi_{s_1} \circ \cdots \circ \Xi_{s_r} \circ \IS_{T_\omega} (\OC_{\tgg}), \nabla^\mu_{\tgg}) \bigr)=0 \quad \text{for $n \neq 0$,}
\]
and that $\Hom_{\Rep(\GB)} \bigl( V, \Hom_{\Coh(\tgg)}(\Xi_{s_1} \circ \cdots \circ \Xi_{s_r} \circ \IS_{T_\omega} (\OC_{\tgg}), \nabla^\mu_{\tgg}) \bigr)$ admits a filtration, as a graded $\OC(\tg^*)$-module, whose associated graded is free. Since this graded module is bounded below, we deduce that $\Hom_{\Rep(\GB)} \bigl( V, \Hom_{\Coh(\tgg)}(\Xi_{s_1} \circ \cdots \circ \Xi_{s_r} \circ \IS_{T_\omega} (\OC_{\tgg}), \nabla^\mu_{\tgg}) \bigr)$ is free over $\OC(\tg^*)$, which finishes the proof.
\end{proof}

\subsection{Costandard objects are coherent sheaves}
\label{ss:costandard-coherent}

Let us first record the following corollary of Proposition~\ref{prop:tensor-tilting}.

\begin{cor}
\label{cor:tensor}
Assume $\GB$ is standard. Let $V$ be a finite dimensional $\GB$-module, and let $\FC$ be in $\ES^{\GB \times \Gm}(\tNC)$.
\begin{enumerate}
\item 
\label{it:tensor-Weyl-standard}
If $V$ admits a Weyl filtration and $\FC$ admits a standard filtration, then $V \otimes \FC$ admits a standard filtration.
\item
\label{it:tensor-good-costandard}
If $V$ admits a good filtration and $\FC$ admits a costandard filtration, then $V \otimes \FC$ admits a costandard filtration.
\end{enumerate}
\end{cor}

\begin{proof}
We prove~\eqref{it:tensor-Weyl-standard}; the proof of~\eqref{it:tensor-good-costandard} is similar. Since $V$ admits a Weyl filtration, it admits a finite right resolution by finite dimensional tilting $\GB$-modules. Similarly, since $\FC$ admits a standard filtration, it admits a right resolution by tilting objects in~$\ES^{\GB \times \Gm}(\tNC)$. It follows that $V \otimes \FC$ belongs to the triangulated subcategory of $D^{\GB \times \Gm}(\tNC)$ generated by objects of the form $M \otimes \TC[i]$ where $M$ is a finite dimensional tilting $\GB$-module, $\TC$ is in $\Tilt(\ES^{\GB \times \Gm}(\tNC))$, and $i \leq 0$. Using Proposition~\ref{prop:tensor-tilting}, we deduce that it belongs to the subcategory generated by non-positive shifts of tilting objects in $\ES^{\GB \times \Gm}(\tNC)$. It follows that
\[
\Hom^j_{D^{\GB \times \Gm}(\tNC)}(V \otimes \FC, \nabla^\mu_{\tNC} \langle n \rangle)=0
\]
for $j>0$, which implies that $V \otimes \FC$ admits a standard filtration.
\end{proof}

Now we can prove that the morphism spaces between standard and costandard exotic sheaves satisfy the following vanishing property. (Note that here we consider morphisms in the derived category of \emph{ordinary} coherent sheaves.)

\begin{prop}
\label{prop:Hom-vanishing}
Assume $\GB$ is standard.
Let $\FC,\GC$ be in $\ES^{\GB \times \Gm}(\tNC)$. If $\FC$ admits a standard filtration and $\GC$ admits a costandard filtration, then we have
\[
\Hom^i_{\Db \Coh(\tNC)}(\FC, \GC) = 0 \qquad \text{for $i \neq 0$.}
\]
\end{prop}

\begin{proof}
Consider the regular $\GB$-module $\FM[\GB]$. Since this module is injective, it admits a good filtration; we choose a filtration $0=M_0 \subset M_1 \subset M_2 \subset \cdots \subset \FM[\GB]$ such that $\FM[\GB]=\cup_{n \geq 0} M_n$, and each $M_n$ is a finite dimensional $\GB$-module which admits a good filtration. 

As in~\eqref{eqn:Hom-Gm-equiv}, the vector space $\Hom^i_{\Db \Coh(\tNC)}(\FC, \GC)$ has a natural structure of $\Gm$-module; 
we denote by $\Hom^i_{\Db \Coh(\tNC)}(\FC, \GC)_j$ the factor where $\Gm$ acts by $t \mapsto t^j$.
Then we have $\Hom^i_{\Db \Coh(\tNC)}(\FC, \GC) \cong \bigoplus_{j \in \ZM} \Hom^i_{\Db \Coh(\tNC)}(\FC, \GC)_j$, and natural isomorphisms
\begin{multline*}
\Hom^i_{\Db \Coh(\tNC)}(\FC, \GC)_j \cong 
\mathsf{H}^i ( \mathsf{Inv}^{\GB \times \Gm}(\Rder\Hom_{\Coh(\tNC)}(\FC,\GC) \otimes \FM[\GB] \langle -j \rangle)) \\ 
\cong \varinjlim_{n \geq 0} \mathsf{H}^i \bigl( \mathsf{Inv}^{\GB \times \Gm}(\Rder\Hom_{ \Coh(\tNC)}(\FC,\GC) \otimes M_n \langle -j \rangle) \bigr) \\
\cong \varinjlim_{n \geq 0} \mathsf{H}^i \bigl( \mathsf{Inv}^{\GB \times \Gm}(\Rder\Hom_{\Coh(\tNC)}(\FC,\GC \otimes M_n \langle -j \rangle) ) \bigr) \\
\cong \varinjlim_{n \geq 0} \Hom^i_{\Db \Coh^{\GB \times \Gm}(\tNC)}(\FC, \GC \otimes M_n \langle -j \rangle).
\end{multline*}
(Here the first isomorphism follows from~\cite[Lemma~I.4.7]{jantzen}, the second one uses the fact that cohomology commutes with direct limits, see~\cite[Lemma~I.4.17]{jantzen}, and the last one uses Proposition~\ref{prop:Hom-equivariant}.) If $i \neq 0$, we have  $\Hom^i_{\Db \Coh^{\GB \times \Gm}(\tNC)}(\FC, \GC \otimes M_n \langle -j \rangle) =0$ for all $n, j \in \ZM$ by Corollary~\ref{cor:tensor}\eqref{it:tensor-good-costandard}; the proposition follows.
\end{proof}

The following corollary of Proposition~\ref{prop:Hom-vanishing} was also obtained (by different methods, and under different assumptions) in~\cite[Proposition~8.7]{arider2}.

\begin{cor}
\label{cor:standard-coherent}
Assume $\GB$ is standard. Then for any $\mu \in \XB$ the complex of coherent sheaves $\nabla^{\mu}_{\tNC}$ has cohomology only in degree $0$.
\end{cor}

\begin{proof}
For $\lambda \in \XB^+$ and $i \in \ZM$, we have
\[
\mathsf{H}^i(\tNC, \nabla^\mu_{\tNC} \otimes_{\OC_{\tNC}} \OC_{\tNC}(\lambda)) \cong \Hom^i_{\Db \Coh(\tNC)}(\OC_{\tNC}(-\lambda), \nabla^\mu_{\tNC}).
\]
Now by Proposition~\ref{prop:Delta}\eqref{it:Delta-antidom} the object $\OC_{\tNC}(-\lambda)$ is standard; using Proposition~\ref{prop:Hom-vanishing} we deduce that
\[
\mathsf{H}^i(\tNC, \nabla^\mu_{\tNC} \otimes_{\OC_{\tNC}} \OC_{\tNC}(\lambda))=0 \qquad \text{if $i \neq 0$.}
\]
As in the proof of Corollary~\ref{cor:negative-cohomology}, this vanishing property for $\lambda$ sufficiently large implies that $\nabla^\mu_{\tNC}$ is concentrated in degree $0$.
\end{proof}

\begin{remark}
In case $p=0$, Corollary~\ref{cor:standard-coherent} can alternatively be deduced from~\cite[Lemma~12]{bezru-tilting} and~\cite[Lemma~30(b)]{bezru-two}. (Note that the arguments in~\cite{bezru-two} are based on ideas similar to ours.)
\end{remark}

\appendix

\section{Equivariant quasi-coherent sheaves}
\label{sec:appendix}

In this appendix
we briefly review the theory of derived categories of equivariant coherent sheaves. We work in a setting which is much more general than what we actually use in the body of the paper, in order to be able to use this appendix as a reference in~\cite{mr} also.

\begin{remark}
Most of the results discussed in this appendix are also proved in~\cite{hashimoto}. However in this reference the author uses a different, more involved, definition of the derived category of equivariant coherent sheaves, and of derived functors between such categories. These definitions are probably equivalent to ours, but we could not find a detailed reference treating this comparison, and hence preferred to review the theory ``from scratch.''
\end{remark}

\subsection{Setting and assumptions}
\label{ss:setting-appendix}

Let $k$ be a Noetherian commutative ring.
All the $k$-schemes we consider below will be tacitly assumed to be Noetherian (in particular, quasi-compact) and quasi-separated. These assumptions ensure that, if $X$ is such a scheme, the natural functor
\[
i_X \colon D \QCoh(X) \to D \Mod(\OC_X)
\]
(where $\Mod(\OC_X)$ denotes the category of all sheaves of $\OC_X$-modules)
is fully-faithful, see~\cite[Corollary~5.5]{bn}. 

Recall that the category $\QCoh(X)$ has enough injective objects, see~\cite[Proposition~II.1.1]{hartshorne}. Therefore,
if $f \colon X \to Y$ is a morphism of schemes, the direct image functor $f_* \colon \QCoh(X) \to \QCoh(Y)$ admits a right derived functor 
\[
\Rder f_* \colon D^+ \QCoh(X) \to D^+ \QCoh(Y).
\]
Moreover, it follows from~\cite[Theorem~II.7.18]{hartshorne} that the following diagram commutes (up to canonical isomorphism), where 
the lower horizontal arrow is the derived functor of the functor $f_* \colon \Mod(\OC_X) \to \Mod(\OC_Y)$:
\[
\xymatrix@C=1.5cm@R=0.5cm{
D^+ \QCoh(X) \ar[r]^-{\Rder f_*} \ar[d]_-{i_X} & D^+ \QCoh(Y) \ar[d]^-{i_Y} \\
D^+ \Mod(\OC_X) \ar[r]^-{\Rder f_*} & D^+\Mod(\OC_Y).
}
\]

We will also assume (again tacitly) that the following property holds for any scheme $X$ we consider: for any $\FC$ in $\Coh(X)$, there exists a locally free $\OC_X$-module of finite rank $\FC'$ and a surjection $\FC' \twoheadrightarrow \FC$.\footnote{This assumption is satisfied if $X$ is affine, or regular, or quasi-projective over $k$; see~\cite[\S 2.3]{thomason} for references.} This assumption implies that for any $\FC$ in $\QCoh(X)$ there exists $\FC'$ in $\QCoh(X)$ which is flat over $\OC_X$ and a surjection $\FC' \twoheadrightarrow \FC$, see~\cite[\S 2.2]{thomason}. It follows that if $f \colon X \to Y$ is a morphism, the functor $f^* \colon \QCoh(Y) \to \QCoh(X)$ admits a left derived functor
\[
\Lder f^* \colon D^- \QCoh(Y) \to D^- \QCoh(X),
\]
and that the following diagram commutes (up to canonical isomorphism), where 
the lower horizontal arrow is the derived functor of $f^* \colon \Mod(\OC_Y) \to \Mod(\OC_X)$:
\[
\xymatrix@C=1.5cm@R=0.5cm{
D^- \QCoh(Y) \ar[r]^-{\Lder f^*} \ar[d]_-{i_Y} & D^- \QCoh(X) \ar[d]^-{i_X} \\
D^- \Mod(\OC_Y) \ar[r]^-{\Lder f^*} & D^-\Mod(\OC_X).
}
\]

Using these remarks and~\cite[Proposition~3.2.3]{lipman}, we deduce that the functors $\Lder f^*$ and $\Rder f_*$ are adjoint, in the sense that for $\FC$ in $D^- \QCoh(Y)$ and $\GC$ in $D^+ \QCoh(X)$ there exists a canonical (in particular, functorial) isomorphism
\begin{equation}
\label{eqn:adjunction}
\Hom_{D \QCoh(X)}(\Lder f^* \FC, \GC) \simto \Hom_{D \QCoh(Y)}(\FC, \Rder f_* \GC).
\end{equation}

Below we will need the following technical results.

\begin{lem}
\label{lem:tensor-flat}
Let $M$ be a flat $k$-module, and let $\FC \in \Coh(X)$ and $\GC \in \QCoh(X)$.
\begin{enumerate}
\item
\label{it:Hom-tensor}
There exists a canonical isomorphism
\[
\Hom_{\QCoh(X)}(\FC, \GC) \otimes_k M \simto
\Hom_{\QCoh(X)}(\FC, \GC \otimes_k M).
\]
\item
\label{it:inj-tensor}
If $\GC$ is injective in $\QCoh(X)$, then $\GC \otimes_k M$ is injective also.
\end{enumerate}
\end{lem}

\begin{proof}
There exists a canonical morphism $\sheafHom_{\OC_X}(\FC, \GC) \otimes_k M \to \sheafHom_{\OC_X}(\FC, \GC \otimes_k M)$. Under our assumptions this morphism is an isomorphism: in fact it suffices to prove this claim when $X$ is affine, and then the same proof as in~\cite[Theorem~7.11]{matsumura} applies. Then, taking global sections and applying the projection formula~\cite[Proposition~II.5.6]{hartshorne}, we deduce~\eqref{it:Hom-tensor}. 

In~\eqref{it:inj-tensor}, let us first assume that $X$ is affine. Then the claim follows from the isomorphism in~\eqref{it:Hom-tensor} and the fact that injectivity can be tested on finitely generated modules (since $\OC(X)$ is Noetherian). To treat the general case, we observe that
any injective object in $\QCoh(X)$ is a direct summand in an injective object which is a finite direct sum of objects of the form $j_* \widetilde{I}$ where $j \colon U \hookrightarrow X$ is an inclusion of an affine open subset, $I$ is an injective $\OC(U)$-module, and $\widetilde{I}$ is the associated quasi-coherent sheaf on $U$. Now we have $(j_* \widetilde{I}) \otimes_k M \cong j_* (\widetilde{I \otimes_k M})$, which reduces the claim to the affine case treated above.
\end{proof}

\subsection{Injective objects}

Let $\HB$ be a flat affine group scheme over $k$.
We denote by $\Rep(\HB)$ the category of (algebraic) representations of $\HB$.

Let $X$ be a $k$-scheme as in~\S\ref{ss:setting-appendix}, endowed with an action of $\HB$. Then we can consider the categories $\Coh^\HB(X)$ and $\QCoh^\HB(X)$ of $\HB$-equivariant coherent and quasi-coherent sheaves on $X$. (See~\cite[\S 1.2]{thomason} for several equivalent definitions of these objects.) Let $a,p \colon \HB \times X \to X$ be the action and projection, respectively. Then we have an ``averaging'' functor
\[
\mathsf{Av}_X \colon \QCoh(X) \to \QCoh^\HB(X)
\]
which sends $\FC$ to $a_* p^* \FC$, with its natural equivariant structure. (Here the functors $p^*$ and $a_*$ are exact, so that we write $p^*$, $a_*$ instead of $\Lder p^*$, $\Rder a_*$.) This functor is exact and right-adjoint to the forgetful functor $\QCoh^\HB(X) \to \QCoh(X)$; in particular it sends injective objects to injective objects. Using this and the fact that $\QCoh(X)$ has enough injectives, one can show that $\QCoh^\HB(X)$ has enough injectives, and moreover that any injective object is a direct summand in an object of the form $\mathsf{Av}_X(\IC)$ where $\IC$ is an injective object of $\QCoh(X)$.

We will freely use the following result, for which we refer to~\cite[Corollary~2.11]{ab}.\footnote{The article~\cite{ab} is written using the language of stacks. However, the proof of their Corollary~2.11 only relies on their Lemma~2.9, which is proved in the setting of equivariant quasi-coherent sheaves in~\cite[Lemma~1.4]{thomason}.}

\begin{lem}
The natural functor $\Db \Coh^\HB(X) \to \Db \QCoh^\HB(X)$ is fully faithful; its essential image consists of objects whose cohomology sheaves are coherent.\qed
\end{lem}

\subsection{Morphisms}
\label{ss:morphisms}


It is well known that if $\FC$ is in $\Coh^\HB(X)$ and $\GC$ is in $\QCoh^\HB(X)$, then the $k$-module $\Hom_{\QCoh(X)}(\FC,\GC)$ naturally carries the structure of an (algebraic) $\HB$-module, and that moreover we have a canonical isomorphism
\begin{equation}
\label{eqn:Hom-equivariant}
\bigl( \Hom_{\QCoh(X)}(\FC,\GC) \bigr)^\HB \cong \Hom_{\QCoh^\HB(X)}(\FC,\GC).
\end{equation}

\begin{lem}
\label{lem:Hom-injective}
Let $\FC$ be in $\QCoh^\HB(X)$, and $\GC$ be an injective object in $\QCoh^\HB(X)$. Then we have
\[
\Ext^n_{\QCoh(X)}(\FC,\GC)=0 \qquad \text{for $n>0$.}
\]
\end{lem}

\begin{proof}
We can assume that $\GC=\mathsf{Av}_X(\IC)$ for some $\IC$ injective in $\QCoh(X)$. Then
\begin{multline*}
\Ext^n_{\QCoh(X)}(\FC,\GC) = \Ext^n_{\QCoh(X)}(\FC,a_* p^* \IC) \cong \Ext^n_{\QCoh(H \times X)}(a^* \FC, p^*\IC) \cong \\
\Ext^n_{\QCoh(H \times X)}(p^* \FC, p^*\IC) \cong \Ext^n_{\QCoh(X)}(\FC, p_*p^*\IC)
\cong \Ext^n_{\QCoh(X)}(\FC, \IC \otimes_k \OC(\HB)).
\end{multline*}
Now $\IC \otimes_k \OC(\HB)$ is injective by Lemma~\ref{lem:tensor-flat}, and the vanishing follows.
\end{proof}

\begin{cor}
For $\FC$ in $\Db \Coh^\HB(X)$, the functor
\[
\Rder\Hom_{\QCoh(X)}(\FC, -) \colon D^+ \QCoh^\HB(X) \to D^+ \Mod(k)
\]
factors canonically through a functor $D^+ \QCoh^\HB(X) \to D^+ \Rep(\HB)$, which we denote similarly.
\end{cor}

\begin{proof}
Fix $\FC^\bullet$ in $C^{\mathrm{b}} \Coh^\HB(X)$ with image $\FC$ in $\Db \Coh^\HB(X)$. Then we consider the functor
\[
\Hom^\bullet_{\QCoh(X)}(\FC^\bullet, -) \colon C^+ \QCoh^\HB(X) \to C^+ \Rep(\HB),
\]
and its derived functor $D^+ \QCoh^\HB(X) \to D^+ \Rep(\HB)$ (which exists since the category $\QCoh^\HB(X)$ has enough injectives). Using Lemma~\ref{lem:Hom-injective} one can easily check that the composition of this derived functor with the forgetful functor $D^+ \Rep(\HB) \to D^+ \Mod(k)$ is the functor $\Rder\Hom_{\QCoh(X)}(\FC, -)$, which proves the claim.
\end{proof}

Consider the derived functor of invariants
\[
\mathsf{Inv}^{\HB} \colon D^+ \Rep(\HB) \to D^+ \Mod(k).
\]

\begin{prop}
\label{prop:Hom-equivariant}
For $\FC$ in $\Db \Coh^{\HB}(X)$ and $\GC$ in $D^+ \QCoh^\HB(X)$, there exists a canonical (in particular, functorial) isomorphism
\[
\mathsf{Inv}^{\HB} \circ \Rder\Hom_{\QCoh(X)}(\FC,\GC) \simto \Rder\Hom_{\QCoh^\HB(X)}(\FC,\GC)
\]
in $D^+ \Mod(k)$.
\end{prop}

\begin{proof}
Using~\eqref{eqn:Hom-equivariant} and basic properties of derived functors, we obtain a functorial morphism
\[
\Rder\Hom_{\QCoh^\HB(X)}(\FC,\GC) \to \mathsf{Inv}^{\HB} \circ \Rder\Hom_{\QCoh(X)}(\FC,\GC)
\]
for $\FC$ in $\Db\Coh^\HB(X)$ and $\GC$ in $D^+ \QCoh^\HB(X)$. To prove that this morphism is an isomorphism it suffices to prove that if $\FC$ is in $\Coh^\HB(X)$ and $\GC$ is an injective object in $\QCoh^\HB(X)$, then we have 
\[
\mathsf{H}^n \bigl( \mathsf{Inv}^\HB(\Hom_{\QCoh(X)}(\FC,\GC)) \bigr) = 0 \qquad \text{for $n>0$.}
\]
To prove this fact we can assume that $\GC=\mathsf{Av}_X(\IC)$ for some $\IC$ injective in $\QCoh(X)$. Then as in the proof of Lemma~\ref{lem:Hom-injective} we have
\[
\Hom_{\QCoh(X)}(\FC,\GC) \cong \Hom_{\QCoh(X)}(\FC, \GC \otimes_k \OC(\HB)) \cong \Hom_{\QCoh(X)}(\FC, \GC) \otimes_k \OC(\HB)
\]
by Lemma~\ref{lem:tensor-flat}\eqref{it:Hom-tensor}.
Then the claim follows from~\cite[Lemma~I.4.7]{jantzen}.
\end{proof}

\subsection{Invariants and base-change}
\label{ss:inv-base-change}

In this subsection we assume that $k$ has finite global dimension. We also make the following assumption: for any algebraic representation $M$ of $\HB$, there exists an algebraic representation $M'$ of $\HB$ which is flat over $k$ and a surjective $\HB$-equivariant morphism $M' \twoheadrightarrow M$.\footnote{This assumption is of course automatic if $k$ is a field. It is also satisfied if $\HB$ is a split reductive group over $k$, see~\cite[\S2.2 \& Corollary~2.9]{thomason}. See~\cite[Section~2]{thomason} for various other sufficient conditions.} We let $k'$ be a Noetherian commutative $k$-algebra, and denote by $\HB'$ the base-change of $\HB$ to $k'$. Then we have a natural forgetful functor $\Rep(\HB') \to \Rep(\HB)$.

\begin{lem}
\label{lem:Inv-for}
If $M$ is an injective object in $\Rep(\HB')$, then we have
\[
\mathsf{H}^n(\Inv^\HB(M))=0 \qquad \text{for $n>0$.}
\]
\end{lem}

\begin{proof}
The cohomology groups $\mathsf{H}^n(\Inv^\HB(M))$ can be computed using the Hochschild complex $C^\bullet(\HB, M)$, see~\cite[\S I.4.16]{jantzen}. Now it is clear that $C^\bullet(\HB, M)$ is canonically isomorphic to $C^\bullet(\HB', M)$, which computes $\Inv^{\HB'}(M)$; therefore its positive cohomology groups vanish since $M$ is injective.
\end{proof}

Let $\For^\HB \colon D^+ \Rep(\HB') \to D^+ \Rep(\HB)$ and $\For \colon D^+ \Mod(k') \to D^+ \Mod(k)$ be the obvious forgetful functors. Then Lemma~\ref{lem:Inv-for}
implies, by standard arguments, that there exists a canonical isomorphism of functors
\begin{equation}
\label{eqn:Inv-for}
\Inv^\HB \circ \For^\HB \simto \For \circ \Inv^{\HB'}.
\end{equation}

Using our assumptions (in particular the fact that $k$ has finite global dimension), we see that the functor $k' \otimes_k (-) \colon \Rep(\HB) \to \Rep(\HB')$ admits a left derived functor
\[
k' \lotimes_k (-) \colon D^+ \Rep(\HB) \to D^+ \Rep(\HB'),
\]
see in particular~\cite[Corollary~I.5.3($\gamma$)]{hartshorne}.
Similarly, we have a derived ``extension of scalars'' functor $k' \lotimes_k (-) \colon D^+ \Mod(k) \to D^+ \Mod(k')$.

\begin{prop}
\label{prop:Inv-tensor}
There exists a canonical isomorphism
\[
k' \lotimes_k \mathsf{Inv}^{\HB}(-) \simto \mathsf{Inv}^{\HB'}( k' \lotimes_k (-))
\]
of functors from $D^+ \Rep(\HB)$ to $D^+ \Mod(k')$.
\end{prop}

\begin{proof}
Let $\iota_\HB \colon D^+ \Mod(k) \to D^+ \Rep(\HB)$ be the natural functor sending a complex to itself with the trivial action, and similarly for $\iota_{\HB'}$. Then $\mathsf{Inv}^{\HB}$ is right-adjoint to $\iota_\HB$, so that there exists a canonical morphism $\iota_\HB \circ \mathsf{Inv}^\HB \to \mathrm{id}$. We deduce a canonical morphism
\[
\iota_{\HB'}  \bigl( k' \lotimes_k (\mathsf{Inv}^\HB(-)) \bigr) \cong k' \lotimes_k (\iota_\HB \circ \mathsf{Inv}^\HB(-)) \to k' \lotimes_k(-).
\]
By adjunction, we deduce a canonical morphism $\eta \colon k' \lotimes_k \mathsf{Inv}^{\HB}(-) \to \mathsf{Inv}^{\HB'}( k' \lotimes_k (-))$.

To prove that $\eta$ is an isomorphism, it suffices to prove that its composition with 
$\For$
is an isomorphism. However, if we fix a finite projective resolution $P^\bullet \to k'$ of $k'$ over $k$, then the composition $\For \bigl( k' \lotimes_k(-) \bigr)$ identifies with the functor induced by the exact functor $M^\bullet \mapsto P^\bullet \otimes_k M^\bullet$ (for $M^\bullet$ a complex of $k$-modules). Similarly, using~\eqref{eqn:Inv-for}, we have
\[
\For \circ \mathsf{Inv}^{\HB'}( k' \lotimes_k (-)) \cong \Inv^\HB \bigl( \For^\HB (k' \lotimes_k (-)) \bigr),
\]
and the functor $\For^\HB (k' \lotimes_k (-))$ also identifies with the functor $M^\bullet \mapsto P^\bullet \otimes_k M^\bullet$ (this time for $M^\bullet$ a complex of representations of $\HB$). Now the functor $\Inv^\HB$ identifies with the functor sending a complex $M^\bullet$ to the total complex of the double complex $C^\bullet(\HB, M^\bullet)$ (with the notation of~\cite[\S I.4.14]{jantzen}), and $\For(\eta)$ is induced by the obvious isomorphism
\[
P^\bullet \otimes_k C^\bullet(\HB, M^\bullet) \simto C^\bullet(\HB, P^\bullet \otimes_k M^\bullet).
\]
Hence it is indeed an isomorphism.
\end{proof}

\subsection{Direct and inverse image functors}

Let now $X$ and $Y$ be two $k$-schemes as in~\S\ref{ss:setting-appendix}, endowed with $\HB$-actions, and consider an $\HB$-equivariant morphism $f \colon X \to Y$. Then we have natural adjoint functors
\[
f^*_\HB \colon \QCoh^\HB(Y) \to \QCoh^\HB(X), \qquad f_*^\HB \colon \QCoh^\HB(X) \to \QCoh^\HB(Y).
\]

\begin{lem}
\label{lem:pushforward-injective}
Let $\FC$ be an injective object of $\QCoh^\HB(X)$. Then we have
\[
\Rder^n f_*(\FC)=0 \qquad \text{for $n>0$.}
\]
\end{lem}

\begin{proof}
We can assume that $\FC=\mathsf{Av}_X(\IC)$ as in the proof of Lemma~\ref{lem:Hom-injective}. Then using the flat base change theorem one can check that
$\Rder f_*(\FC) \cong \mathsf{Av}_Y(\Rder f_*(\IC))$
(where we have omitted the forgetful functors, and written $\mathsf{Av}_Y$ for the derived functor of the exact functor $\mathsf{Av}_Y$.) The claim follows.
\end{proof}

\begin{prop}
\label{prop:direct-image}
The functor $f_*^\HB \colon \QCoh^\HB(X) \to \QCoh^\HB(Y)$ admits a right derived functor $\Rder f^\HB_* \colon D^+ \QCoh^\HB(X) \to D^+ \QCoh^\HB(Y)$. Moreover the following diagram commutes, where the vertical arrows are the natural forgetful functors:
\begin{equation}
\label{eqn:diag-direct-image}
\vcenter{
\xymatrix@C=1.5cm@R=0.5cm{
D^+ \QCoh^\HB(X) \ar[r]^-{\Rder f^\HB_*} \ar[d] & D^+ \QCoh^\HB(Y) \ar[d] \\
D^+ \QCoh(X) \ar[r]^-{\Rder f_*} & D^+ \QCoh(Y).
}
}
\end{equation}
If $f$ and $g$ are composable $\HB$-equivariant morphisms, then there exists a canonical isomorphism
$\Rder (f \circ g)^\HB_* \cong \Rder f^\HB_* \circ \Rder g^\HB_*$.
\end{prop}

\begin{proof}
Since $\QCoh^\HB(X)$ has enough injective objects, the derived functor $\Rder f^\HB_*$ exists. The commutativity of the diagram follows from Lemma~\ref{lem:pushforward-injective}. Then if $f \colon Y \to Z$ and $g \colon X \to Y$ are composable morphisms, since $(f \circ g)^\HB_* \cong f^\HB_* \circ g^\HB_*$, by the general theory of derived functors there exists a canonical morphism $\xi \colon \Rder(f \circ g)^\HB_* \to \Rder f^\HB_* \circ \Rder g^\HB_*$. Since $\Rder(f \circ g)_* \cong \Rder f_* \circ \Rder g_*$, the image of $\xi$ under the forgetful functor $D^+ \QCoh^\HB(Z) \to D^+ \QCoh(Z)$ is an isomorphism. It follows that $\xi$ itself is an isomorphism.
\end{proof}

\begin{remark}
It follows from the commutativity of~\eqref{eqn:diag-direct-image} that the functor $\Rder f^\HB_*$ restricts to a functor $\Db \QCoh^\HB(X) \to \Db \QCoh^\HB(Y)$; see~\cite[Proposition~3.9.2]{lipman}.
\end{remark}

From now on we assume that for every $\FC$ in $\QCoh^\HB(Y)$, there exists $\GC \in \QCoh^\HB(Y)$ which is flat over $\OC_Y$ and a surjection $\GC \twoheadrightarrow \FC$.\footnote{See~\cite{thomason} for various sufficient conditions for this assumption to hold. In particular, by~\cite[Lemma~2.10]{thomason}, it holds if $\HB$ is split reductive and $X$ is normal and quasi-projective over $k$.} Under this assumption, the proof of the following proposition is similar to that of Proposition~\ref{prop:direct-image}.

\begin{prop}
\label{prop:inverse-image}
The functor $f^*_\HB \colon \QCoh^\HB(Y) \to \QCoh^\HB(X)$ admits a left derived functor $\Lder f^*_\HB \colon D^- \QCoh^\HB(Y) \to D^- \QCoh^\HB(X)$. Moreover the following diagram commutes, where the vertical arrows are the natural forgetful functors:
\begin{equation}
\label{eqn:diag-inverse-image}
\vcenter{
\xymatrix@C=1.5cm@R=0.5cm{
D^- \QCoh^\HB(Y) \ar[r]^-{\Lder f_\HB^*} \ar[d] & D^- \QCoh^\HB(X) \ar[d] \\
D^- \QCoh(Y) \ar[r]^-{\Lder f^*} & D^- \QCoh(X).
}
}
\end{equation}
If $f$ and $g$ are composable $\HB$-equivariant morphisms, then there exists a canonical isomorphism
$\Lder (f \circ g)_\HB^* \cong \Lder g_\HB^* \circ \Lder f_\HB^*$.\qed
\end{prop}


Using similar arguments one can check that the bifunctor $(-) \otimes_{\OC_X} (-)$ admits a left derived functor
\[
(-) \lotimes_{\OC_X} (-) \colon D^- \QCoh^\HB(X) \times D^- \QCoh^\HB(X) \to D^- \QCoh^\HB(X)
\]
which is compatible with inverse image functors in the obvious sense.

\begin{remark}
\begin{enumerate}
\item
In the body of the paper, we use the notation $\Lder f^*$ instead of $\Lder f_\HB^*$, $\Rder f_*$ instead of $\Rder f^\HB_*$. This is justified by the commutativity of~\eqref{eqn:diag-direct-image} and~\eqref{eqn:diag-inverse-image}.
\item
In Proposition~\ref{prop:inverse-image}, if
$f$ is flat then the functor $\Lder f^*_\HB$ is defined on the whole of $D \QCoh^\HB(Y)$.
\end{enumerate}
\end{remark}


\subsection{Adjunction, projection formula, and flat base change}


\begin{lem}
\label{lem:Ext-flat-injective}
Let $\FC$ be an object of $\QCoh^\HB(Y)$ which is flat over $\OC_Y$, and let $\GC$ be an injective object of $\QCoh^\HB(X)$. Then we have
\[
\Ext^n_{\QCoh^\HB(Y)}(\FC, f^\HB_*(\GC)) = \Ext^n_{\QCoh(Y)}(\FC, f^\HB_*(\GC)) = 0 \qquad \text{for $n>0$.}
\]
\end{lem}

\begin{proof}
We can assume that $\GC=\mathsf{Av}_X(\IC)$ as in the proof of Lemma~\ref{lem:Hom-injective}.
Then (as in the proof of Lem\-ma~\ref{lem:pushforward-injective}) we have $f_*^\HB(\GC) \cong \mathsf{Av}_Y(f_* \IC)$, so that
\[
\Ext^n_{\QCoh^\HB(Y)}(\FC, f^\HB_*(\GC)) \cong \Ext^n_{\QCoh^\HB(Y)}(\FC, \mathsf{Av}_Y(f_*\IC)) \cong \Ext^n_{\QCoh(Y)}(\FC, f_*\IC).
\]
Then using the adjunction $(\Lder f^*, \Rder f_*)$ (see~\eqref{eqn:adjunction}) we have
\[
\Ext^n_{\QCoh(Y)}(\FC, f_*\IC) \cong \Hom^n_{D\QCoh(Y)}(\FC, \Rder f_*\IC) \cong \Hom^n_{D\QCoh(X)}(\Lder f^*\FC, \IC),
\]
and the right-hand side vanishes for $i>0$ since $\Lder f^* \FC$ is concentrated in degree $0$ and $\IC$ is injective. This proves the first claim.

Using Proposition~\ref{prop:direct-image} we have
\begin{multline*}
\Ext^n_{\QCoh(Y)}(\FC, f^\HB_*(\GC)) \cong \Hom^n_{D\QCoh(Y)}(\FC, \Rder f_*(\GC)) \\
\cong \Hom^n_{D\QCoh(X)}(\Lder f^*\FC, \GC) \cong \Ext^n_{\QCoh(X)}(f^*\FC, \GC).
\end{multline*}
Then the claim follows from Lemma~\ref{lem:Hom-injective}.
\end{proof}

\begin{prop}
\label{prop:adj-proj}
Let $f\colon X \to Y$ be an $\HB$-equivariant morphism.
\begin{enumerate}
\item
\label{it:adjunction}
(Adjunction)
For $\FC$ in $D^- \QCoh^\HB(Y)$ and $\GC$ in $D^+ \QCoh^\HB(X)$ there exists a canonical and functorial isomorphism
\[
\Hom_{D \QCoh^\HB(X)}(\Lder f_\HB^* \FC, \GC) \cong \Hom_{D\QCoh^\HB(Y)}(\FC, \Rder f^\HB_*\GC).
\]
Moreover the following diagram commutes, where the vertical arrows are induced by the forgetful functors (see Propositions~{\rm \ref{prop:direct-image}} and~{\rm \ref{prop:inverse-image}}):
\[
\xymatrix@C=1.5cm@R=0.5cm{
\Hom_{D \QCoh^\HB(X)}(\Lder f_\HB^* \FC, \GC) \ar[r]^-{\sim} \ar[d] & \Hom_{D\QCoh^\HB(Y)}(\FC, \Rder f^\HB_*\GC) \ar[d] \\
\Hom_{D \QCoh(X)}(\Lder f^* \FC, \GC) \ar[r]_-{\sim}^-{\eqref{eqn:adjunction}} & \Hom_{D\QCoh(Y)}(\FC, \Rder f_*\GC).
}
\]
\item
\label{it:projection}
(Projection formula)
Let $\FC \in \Db \QCoh^\HB(Y)$, and $\GC \in \Db \QCoh^\HB(X)$. Assume\footnote{This assumption is automatic e.g.~if $X$ and $Y$ are regular schemes.} that the complex $\Lder f^*_\HB(\FC) \lotimes_{\OC_X} \GC$ belongs to $D^b \QCoh^\HB(X)$. Then there exists a canonical isomorphism
\[
\FC \lotimes_{\OC_Y} \Rder f_*^\HB(\GC) \simto \Rder f_*^\HB(\Lder f^*_\HB(\FC) \lotimes_{\OC_X} \GC).
\]
\item
\label{it:flat-base-change}
(Flat base change)
If $v \colon Z \to Y$ is a flat $\HB$-equivariant morphism, and if $v' \colon X \times_Y Z \to X$ and $f' \colon X \times_Y Z \to Z$ are the morphisms obtained by base change, for any $\FC \in \Db \QCoh^\HB(X)$ there exists a canonical isomorphism
\[
v^*_\HB \circ \Rder f^\HB_* \FC \simto \Rder (f')^\HB_* \circ (v')_\HB^* \FC.
\]
\end{enumerate}
\end{prop}

\begin{proof}
In~\eqref{it:adjunction} we can assume that $\FC$ is a bounded above complex of objects which are flat over $\OC_Y$, and that $\GC$ is a bounded below complex of injective objects. Then the claims follow from the adjunction $(f_\HB^*, f_*^\HB)$ and Lemma~\ref{lem:Ext-flat-injective}. In~\eqref{it:projection} the morphism is constructed using adjunction, and to prove that it is an isomorphism it suffices to check this property for its image under the forgetful functor $D \QCoh^\HB(Y) \to D\QCoh(Y)$, which follows from the usual projection formula (see~\cite[Proposition~3.9.4]{lipman}). The proof of~\eqref{it:flat-base-change} is similar.
\end{proof}

Propositions~\ref{prop:direct-image}, \ref{prop:inverse-image} and~\ref{prop:adj-proj} provide what is necessary to have a theory of Fourier--Mukai transforms in the equivariant setting.

\end{document}